\documentclass[a4paper,12pt,reqno]{amsart}

\usepackage{amsmath,amsfonts,amsthm,amssymb,color}
\usepackage[T1]{fontenc}

\usepackage{pdfsync}

\newcommand{\he}{\hat\epsilon}

\newcommand{\var}{\mathbf{Var}}

\newcommand{\iot}{\int_{0}^{t}}

\newcommand{\ist}{\int_{s}^{t}}
 
\newcommand{\ott}{[0,T]}
\newcommand{\ou}{[0,1]}

\newcommand{\id}{\mbox{id}}

\newcommand{\lot}{[\ell_1,\ell_2]}
\newcommand{\norm}[1]{\lVert #1\rVert}
\newcommand{\xd}{\mathbf{x}^{\mathbf{2}}}

\newcommand{\bb}{\mathbf{B}}

\newcommand{\be}{\mathbf{E}}
\newcommand{\bk}{\mathbf{k}}
\newcommand{\bp}{\mathbf{P}}
\newcommand{\bv}{\mathbf{V}}
\newcommand{\bx}{{\bf x}}

\newcommand{\xdt}{{\bf \tilde{x}^{2}}}
\newcommand{\der}{\delta}

\newcommand{\1}{{\bf 1}}
\newcommand{\2}{{\bf 2}}

\newcommand{\cZ}{{\mathcal Z}}

\def\EE{\mathbf{E}}
\def\RR{\mathbb{R}}

\def\de{{\delta}}

\def\De{{\Delta}}

\def\de{{\delta}}
\def\De{{\Delta}}

\def \eref#1{\hbox{(\ref{#1})}}

\newcommand{\D}{\mathbb D}
\newcommand{\R}{\mathbb R}
\newcommand{\N}{\mathbb N}

\newcommand{\cb}{\mathcal B}
\newcommand{\cac}{\mathcal C}

\newcommand{\ce}{\mathcal E}
\newcommand{\cf}{\mathcal F}
\newcommand{\ch}{\mathcal H}
\newcommand{\ci}{\mathcal I}
\newcommand{\cj}{\mathcal J}
\newcommand{\cl}{\mathcal L}
\newcommand{\cn}{\mathcal N}
\newcommand{\cq}{\mathcal Q}
\newcommand{\cs}{\mathcal S}
\newcommand{\cv}{\mathcal V}

\newcommand{\al}{\alpha}
\newcommand{\ep}{\varepsilon}

\newcommand{\ga}{\gamma}

\newcommand{\ka}{\kappa}
\newcommand{\la}{\lambda}
\newcommand{\laa}{\Lambda}

\newcommand{\oom}{\Omega}

\newcommand{\si}{\sigma}

\newcommand{\vp}{\varphi}

\newcommand{\lp}{\left(}
\newcommand{\rp}{\right)}
\newcommand{\lc}{\left[}
\newcommand{\rc}{\right]}
\newcommand{\lcl}{\left\{}
\newcommand{\rcl}{\right\}}
\newcommand{\lln}{\left|}
\newcommand{\rrn}{\right|}
\newcommand{\lla}{\left\langle}
\newcommand{\rra}{\right\rangle}

\newtheorem{theorem}{Theorem}[section]

\newtheorem{corollary}[theorem]{Corollary}

\newtheorem{definition}[theorem]{Definition}

\newtheorem{hypothesis}[theorem]{Hypothesis}
\newtheorem{lemma}[theorem]{Lemma}
\newtheorem{notation}[theorem]{Notation}

\newtheorem{proposition}[theorem]{Proposition}

\theoremstyle{remark}

\theoremstyle{remark}
\newtheorem{example}[theorem]{Example}

\newcommand{\bean}{\begin{eqnarray*}}
\newcommand{\eean}{\end{eqnarray*}}
\newcommand{\ben}{\begin{enumerate}}
\newcommand{\een}{\end{enumerate}}
\newcommand{\beq}{\begin{equation}}
\newcommand{\eeq}{\end{equation}}

\begin{document}

\title[Nilpotent RDEs]{Smooth density for some nilpotent \\ rough differential equations}

\date{\today}

\author{Yaozhong Hu  \and Samy Tindel}
\date{\today}
\begin{abstract}
In this note, we provide a non trivial example of differential equation driven by a fractional Brownian motion  with Hurst parameter $1/3<H<1/2$, whose solution admits a smooth density with respect to Lebesgue's measure. The result is obtained through the use of an explicit representation of the solution  when the vector fields of the equation are nilpotent, plus a Norris type lemma in the rough paths context.
\end{abstract}

\address{Yaozhong Hu, Department of Mathematics, University of Kansas, Lawrence, Kansas, 66045 USA.}
\email{hu@math.ku.edu}

\address{Samy Tindel, Institut {\'E}lie Cartan Nancy, Universit\'e de Nancy 1, B.P. 239,
54506 Vand{\oe}uvre-l{\`e}s-Nancy Cedex, France.}
\email{tindel@iecn.u-nancy.fr}

\thanks{S. Tindel is partially supported by the ANR grant ECRU}

\subjclass[2000]{60H07, 60H10, 65C30}
\date{\today}
\keywords{fractional Brownian motion, rough paths, Malliavin calculus, hypoellipticity}

\maketitle

\section{Introduction}

Let $B=(B^1,\ldots,B^d)$ be a $d$-dimensional fractional Brownian motion with Hurst parameter $1/3<H<1/2$, defined on a complete probability space $(\oom,\cf,\bp)$. Remind that this means that all the component $B^i$ of $B$ are independent centered Gaussian processes with covariance
\begin{equation}\label{eq:def-cov-fbm}
R_H(t,s):=\be\lc  B^i_t \, B^i_s \rc=\frac 12 (s^{2H}+t^{2H}-|t-s|^{2H}).
\end{equation}
In particular, the paths of $B$ are $\ga$-H\"older continuous for all $\ga \in (0,H)$.
This paper is concerned with a class of $\R^m$-valued stochastic differential equations driven by $B$, of the form
\begin{equation}\label{eq:sde-intro}
dy_t=\sum_{i=1}^{d} V_i(y_t) \, dB_t^i, \quad t\in\ott, \quad y_0=a,
\end{equation}
where $T>0$ is a fixed time horizon, $a\in\R^m$ stands for a given initial condition and $(V_1,\ldots,V_d)$ is a family of smooth vector fields of $\R^m$.

\smallskip

Stochastic differential systems driven by fractional Brownian motion have been the object of intensive studies during the past decade, both for their theoretical interest and for the wide range of application they open, covering for instance  finance \cite{Gua,WTT} or  biophysics \cite{KS,OTHP} situations. The first aim in the theory has thus been to settle some reasonable tools allowing to solve equations of type (\ref{eq:sde-intro}). This has been achieved, when the Hurst parameter $H$ of the underlying fBm is $>1/2$, thanks to methods of fractional integration \cite{NR,Za}, or simply by means of Young type integration (see e.g \cite{Gu}). When one moves to more irregular cases, namely $H<1/2$, the standard method by now in order to solve equations like (\ref{eq:sde-intro}) relies on rough paths considerations, as explained for instance in~ \cite{FV,Gu,LQ}.

\smallskip

A second natural step in the study of fractional differential systems consists in establishing some properties about their probability law. Some substitute for the semigroup property governing $\cl(y_t)$ in the Markovian case (namely when $H=1/2$) have been given in \cite{BC,NNRT}, in terms of asymptotic expansions in a neighborhood of $t=0$. nSome considerable efforts have also been 
made  in order to analyze the density of $\cl(y_t)$ with respect to Lebesgue measure. To that respect, in the regular case $H>1/2$ the situation is rather clear: the existence of a density is shown in \cite{NS} under some standard nondegeneracy conditions, the smoothness of the density is established in \cite{HN} under elliptic conditions on the coefficients, and this result is extended to the hypoelliptic case in \cite{BH}. In all, this set of results replicates what has been obtained for the usual Brownian motion, at the price of huge technical complications.

\smallskip

In the irregular case $H<1/2$, the picture is far from being so complete. Indeed, the existence part of the density results have been thoroughly studied under elliptic and H\"ormander conditions (see \cite{CF,FV} for a complete review). However, when one wishes to establish the smoothness of the density, some strong moment assumptions on the inverse of the Malliavin derivative of $y_t$ are usually required. These moment estimates are still an important open question in the field, as well as the smoothness of density for random variables like $y_t$.

\smallskip

The current paper proposes to make a step in this direction, and we wish to prove that $\cl(y_t)$ can be decomposed as $p_t(z) \, dz$ for a smooth function $p_t$ in some special non trivial examples of equation (\ref{eq:sde-intro}). Namely, we will handle in the sequel the case of nilpotent vector fields $V_1,\ldots,V_d$ (see Hypothesis \ref{hyp:nilpotent-V} for a precise description), and in this context we shall derive the following density result:
\begin{theorem}\label{thm:smooth-density}
Suppose that the vector fields $V_i$, $1=1, 2, \ldots, d$ are smooth with all derivatives bounded, and that they are $n$-nilpotent in the sense that their Lie brackets of order $n$ vanish for some positive integer $n$.  We also assume that $V_1,\ldots,V_d$ satisfy H\"ormander's hypoelliptic condition (their Lie brackets generate $\R^m$ at any point $x\in\R^m$), and that all the Lie brackets of order greater or equal to $2$ are  constant.  Then for all $t>0$ the probability law of 
the random variable $y_t$, defined by  (\ref{eq:sde-intro})   admits a smooth density with respect to Lebesgue measure.
\end{theorem}
Notice that the hypoelliptic assumption is quite natural in our context. Indeed, it would certainly be too restrictive to consider a family of vector fields $V_1,\ldots,V_d$ being nilpotent and elliptic at the same time. Moreover, some interesting examples of equations satisfying our standing assumptions will be given below. It should be stressed  however that the basic aim of this article is to prove that smoothness of density results can be obtained for rough differential equations driven by a fractional Brownian motion in some specific situations, even if the general hypoelliptic case is still an important open problem. We refer to \cite{BOT} for another case, based on skew-symmetric properties, where a similar theorem holds true.

\smallskip

In order to prove Theorem \ref{thm:smooth-density}, two main ingredients have to be highlighted:

\smallskip

\noindent
\emph{(i)} Working under the nilpotent assumptions described above enables to use a Strichartz type representation for the solution to our equation, given in terms of a finite chaos expansion. This allows to derive some bounds for the moments of both $y_t$ and its Malliavin derivative, which is the main missing tool on the way to smoothness of density results for rough differential equations in the general case.

\smallskip

\noindent
\emph{(ii)} With the integrability of Malliavin derivative in hand, we shall follow the standard probabilistic way to prove smoothness of density under H\"ormander's conditions, for which we refer to \cite{Ha11,Ma,No}. To this purpose, the second main ingredient is a Norris type lemma, which has to be extended (in the rough path context) to \emph{controlled processes}. It should be mentioned at this point that a similar result has been proven recently (an independently) in \cite{HP}.

\smallskip

\noindent
These two ingredients will be developed in the remainder of the article.

\smallskip

Here is how our article is structured: Some preliminaries on rough differential equations and fractional Brownian motion are given in Section \ref{sec:alg-integration}. Section \ref{sec:norris-lemma} is devoted to the proof of our Norris type lemma for controlled processes in the sense of \cite{Gu}. Finally, Malliavin calculus tools and their application to density results for the random variable $y_t$ are presented at Section \ref{sec:malliavin-calculus}.

\smallskip

\noindent
\textbf{Notation:} In the remainder of the article, $c,c_1,c_2$ will stand for generic positive constants which may change from line to line. We also write $a\lesssim b$ (resp. $a\asymp b$) when $a\le c \, b$ (resp. $a=c\, b$) for a universal constant $c$.

\section{Rough differential equations and fractional Brownian motion}
\label{sec:alg-integration}

Generalized integrals will be needed in the sequel in order to define and solve equations of the form (\ref{eq:sde-intro}), and also to get an equivalent of Norris lemma in our context. Though all those elements might be obtained within the landmark of usual rough paths setting \cite{FV,LQ} we have chosen here to work with the algebraic integration framework, which (from our point of view) is more amenable to handy calculations.

\smallskip

In this section, we recall thus the main concepts of algebraic integration. Namely, we state the definition of the spaces of increments, of the operator $\delta$, and its inverse called $\Lambda$ (or sewing map). We also recall some elementary but useful algebraic relations on the spaces of increments. The interested reader is sent to \cite{Gu} for a complete account on the topic, or to \cite{DT,GT} for a more detailed summary.

\subsection{Increments}\label{sec:incr}

The extended integral we deal
with is based on the notion of increments, together with an
elementary operator $\der$ acting on them.

The notion of increment can be introduced in the following way:  for two arbitrary real numbers
$\ell_2>\ell_1\ge 0$, a vector space $V$, and an integer $k\ge 1$, we denote by
$\cac_k( [\ell_1,\ell_2]; V)$ the set of continuous functions $g : [\ell_1,\ell_2]^{k} \to V$ such
that $g_{t_1 \cdots t_{k}} = 0$
whenever $t_i = t_{i+1}$ for some $i\in \{0, \ldots,  k-1\}$.
Such a function will be called a
\emph{$(k-1)$-increment}, and we will
set $\cac_*( [\ell_1,\ell_2];V)=\cup_{k\ge 1}\cac_k( [\ell_1,\ell_2]; V)$.
To simplify the notation, we will write $\cac_k(V)$, if there is no ambiguity about $ [\ell_1,\ell_2] $.

The operator $\der$
is  an operator acting on
$k$-increments,
and is defined as follows on $\cac_k(V)$:
\begin{equation}
  \label{eq:coboundary}
\delta : \cac_k(V) \to \cac_{k+1}(V), \qquad
(\delta g)_{t_1 \cdots t_{k+1}} = \sum_{i=1}^{k+1} (-1)^i g_{t_1
  \cdots \hat t_i \cdots t_{k+1}} ,
\end{equation}
where $\hat t_i$ means that this particular argument is omitted.
Then a fundamental property of $\der$, which is easily verified,
is that
$\delta \delta = 0$, where $\delta \delta$ is considered as an operator
from $\cac_k(V)$ to $\cac_{k+2}(V)$.
 We will denote $\cZ\cac_k(V) = \cac_k(V) \cap \text{Ker}\delta$
and $\cb \cac_k(V) =\cac_k(V) \cap \text{Im}\delta$.

\smallskip

Some simple examples of actions of $\der$,
which will be the ones we will really use throughout the article,
 are obtained by letting
$g\in\cac_1(V)$ and $h\in\cac_2(V)$. Then, for any $t,u,s\in\lot$, we have
\begin{equation}
\label{eq:simple_application}
  (\der g)_{st} = g_t - g_s
\quad\mbox{ and }\quad
(\der h)_{sut} = h_{st}-h_{su}-h_{ut}.
\end{equation}

\smallskip

Our future discussions will mainly rely on
$k$-increments with $k =2$ or $k=3$, for which we will use some
analytical assumptions. Namely,
we measure the size of these increments by H\"older norms
defined in the following way: for $f \in \cac_2(V)$ let
\begin{equation}\label{eq:def-holder-C2}
\norm{f}_{\mu} = \sup_{s,t\in\lot}\frac{|f_{st}|}{|t-s|^\mu}
\quad\mbox{and}\quad
\cac_2^\mu(V)=\lcl f \in \cac_2(V);\, \norm{f}_{\mu}<\infty  \rcl.
\end{equation}
Using this notation, we define in a natural way $\cac_1^\mu(V)=\{ f \in \cac_1(V); \, \norm{\der f}_\mu < \infty \}$. In the sequel, we also handle norms including supremums, of the form
\begin{equation}\label{eq:def-holder-sup}
\|f\|_{\mu,\infty}= \|f\|_{\mu} + \|f\|_{\infty}, 
\quad\mbox{and}\quad
\cac_1^{\mu,0}(V)=\{ f \in \cac_1(V); \,  \|f\|_{\mu,\infty}  < \infty \}.
\end{equation}
In the same way, for $h \in \cac_3(V)$ we set
\begin{eqnarray}
  \label{eq:normOCC2}
  \norm{h}_{\gamma,\rho} &=& \sup_{s,u,t\in\lot}
\frac{|h_{sut}|}{|u-s|^\gamma |t-u|^\rho},\\
\norm{h}_\mu &= &
\inf\left \{\sum_i \norm{h_i}_{\rho_i,\mu-\rho_i} ;\, h  =\sum_i h_i,\, 0 < \rho_i < \mu \right\} ,\nonumber
\end{eqnarray}
where the last infimum is taken over all sequences $\{h_i, \, i \in \mathbb{N}\} \subset \cac_3(V) $ such that $h
= \sum_i h_i$ and over all choices of the numbers $\rho_i \in (0,\mu)$.
Then  $\norm{\cdot}_\mu$ is easily seen to be a norm on $\cac_3(V)$, and we define
$$
\cac_3^\mu(V):=\lcl h\in\cac_3(V);\, \norm{h}_\mu<\infty \rcl.
$$
Eventually,
let $\cac_3^{1+}(V) = \cup_{\mu > 1} \cac_3^\mu(V)$,
and note  that the same kind of norms can be considered on the
spaces $\cZ \cac_3(V)$, leading to the definition of the  spaces
$\cZ \cac_3^\mu(V)$ and $\cZ \cac_3^{1+}(V)$. In order to avoid ambiguities, we sometimes denote  in the following  by $\cn[  \, \cdot \,  ; \cac_j^\kappa]$ the $\kappa$-H\"older norm on the space $\cac_j$, for $j=1,2,3$. For $\zeta\in\cac_j(V)$, we also set $\mathcal{N}[\zeta;\mathcal{C}_{j}^{0}(V)]=\sup_{s\in[\ell_1; \ell_2]^j}\lVert \zeta_s\rVert_{V}$.

\smallskip

The invertibility of $\der$ under H\"older regularity conditions is an essential tool for the construction of our generalized integrals, and can be summarized as follows:

\begin{theorem}[The sewing map] \label{prop:Lambda}
Let $\mu >1$. For any $h\in \cZ \cac_3^\mu( V)$, there exists a unique $\Lambda h \in \cac_2^\mu(V)$ such that $\der( \Lambda h )=h$. Furthermore,
\begin{eqnarray} \label{contraction}
\norm{ \Lambda h}_\mu \leq \frac{1}{2-2^{\mu}}\, \cn[h;\, \cac_3^{\mu}(V)].
\end{eqnarray}
 This gives rise to a  continuous linear map $\laa:  \cZ \cac_3^\mu( V) \rightarrow \cac_2^\mu(V)$ such that $\der \laa =\id_{ \cZ \cac_3^\mu( V)}$.
\end{theorem}

\begin{proof}
The original proof of this result can be found in \cite{Gu}. We refer to \cite{DT,GT} for two simplified versions.

\end{proof}

The sewing map creates a first link between the structures we just introduced and the problem of integration of irregular functions:

\begin{corollary}[Integration of small increments]
\label{cor:integration}
For any 1-increment $g\in\cac_2 (V)$ such that $\der g\in\cac_3^{1+}$,
set
$
h = (\id-\Lambda \delta) g
$.
Then, there exists $f \in \cac_1(V)$ such that $h=\der f$ and
$$
\delta f_{st} = \lim_{|\Pi_{st}| \to 0} \sum_{i=0}^n g_{t_{i} t_{i+1}},
$$
where the limit is over any partition $\Pi_{st} = \{t_0=s,\dots,
t_n=t\}$ of $[s,t]$ whose mesh tends to zero. The
1-increment $\delta f$ is the indefinite integral of the 1-increment $g$.
\end{corollary}

\smallskip

We also need some product rules for the operator $\delta$. For this
recall the following convention:
for  $g\in\cac_n( \lot \R^{l,d})$ and $h\in\cac_m( \lot ;\R^{d,p}) $ let  $gh$
be the element of $\cac_{n+m-1}( \lot ;\R^{l,p})$ defined by
\begin{equation}\label{cvpdt}
(gh)_{t_1,\dots,t_{m+n-1}}=g_{t_1,\dots,t_{n}} h_{t_{n},\dots,t_{m+n-1}}
\end{equation}
for $t_1,\dots,t_{m+n-1}\in \lot$. With this notation, the following elementary rule holds true:

\begin{proposition}\label{difrul}
Let $g\in\cac_2 (\lot;\R^{l,d})$ and $h\in\cac_1(\lot;\R^d)$. Then
$gh$ is an element of $\cac_2( \lot ;\R^l)$ and
$\der (gh) = \der g\, h  - g \,\der h.$
\end{proposition}

\smallskip

\subsection{Random differential equations}\label{sec:rdes}
One of the main appeals of the algebraic integration theory is that
differential equations driven by   a $\ga$-H\"older signal $x$ can be defined and solved rather quickly in this setting. In the case of an H\"older exponent $\ga>1/3$, the required structures  are just the notion of \textit{controlled processes} and the L\'evy area based on $x$.

\smallskip

Indeed, recall that we wish to consider an equation of the form
\begin{align} \label{de}
dy_t=\sum_{i=1}^{d}V_i(y_t) \, dx_t^{i},
\quad t \in \lc 0, T \rc, \qquad  y_0=a,
\end{align}
where $a$ is a given initial condition in $\R^m$, $x$ is an element of $\cac_1^{\ga}([0,T];\, \R^d)$, and $(V_1,\ldots,V_d)$ is a family of smooth vector fields of $\R^m$. Then it is  natural that the increments of a candidate  for a solution to (\ref{de}) should be controlled by the increments of $x$ in the following way:
\begin{definition}
\label{def:controlled-process}
Let $z$ be a path in $\cac_1^\ka(\R^m)$ with $1/3<\ka\le\ga$, and set $\der x:=\bx^{\1}$.
We say that $z$ is a weakly controlled path based on $x$ if
$z_0=a$ with $a\in\R^{m}$,
and $\der z\in\cac_2^\ka(\R^{m})$ has a decomposition $\der z=\zeta \bx^{1}+ r$,
that is, for any $s, t\in[0,T]$,
\begin{equation}
\label{f3.0.4}
\der z_{st}=\zeta_s^{j} \, \bx^{\1,j}_{st} + r_{st},
\end{equation}
where we have used the summation over repeated indices convention, and
with $\zeta^{1},\ldots,\zeta^{d}\in\cac_1^\ka(\R^{m})$, as well as $r\in\cac_2^{2\ka}(\R^{m})$.
\end{definition}

\noindent
The space of weakly controlled
paths will be denoted by $\cq^x_{\ka,a}(\R^{m})$, and a process
$z\in\cq^x_{\ka,a}(\R^{m})$ can be considered in fact as a couple
$(z,\zeta)$.  The space $\cq^x_{\ka,a}(\R^{m})$ is endowed with a natural semi-norm given by
\begin{equation}\label{f3.0.5}
\cn[z;\cq^x_{\ka,a}(\R^{m})]   = \cn[z;\cac_1^{\ka}(\R^{m})]
+ \sum_{j=1}^{d} \cn[\zeta^{j};\cac_1^{\ka,0}(\R^{m})] 
+\cn[r;\cac_2^{2\ka}(\R^{m})],
\end{equation}
where the quantities $\cn[g;\cac_j^{\ka}]$ have been defined in Section \ref{sec:incr}.
For the L\'evy area associated to $x$ we assume the following structure:
\begin{hypothesis}\label{h1}
The path $x:[0,T]\to\R^{d}$ is
$\ga$-H\"older continuous with $\frac{1}{3}<\ga\le 1$ and  admits a so-called L\'evy area,
that is, a process
$\xd\in\cac_2^{2\ga}(\R^{d,d})$, which  satisfies
$\der\xd=\bx^{\1}\otimes \bx^{\1}$, namely
\begin{equation*}
\der\bx^{\2,ij}_{sut} =[\bx^{\1,i}]_{su} [\bx^{\1,j}]_{ut},
\end{equation*}
for any $s,u,t\in[0,T]$ and $i,j\in\{1,\ldots,d  \}$.
\end{hypothesis}

\smallskip

To illustrate the idea behind the construction of the generalized  integral
assume that the paths $x$ and $z$ are smooth  and also for simplicity that $d=m=1$. Then the Riemann-Stieltjes integral of $z$ with respect to $x$ is well defined and
we have
$$
\ist z_udx_u = z_s(x_t-x_s) + \ist(z_u-z_s)dx_u  = z_s \bx^{\1}_{st} + \ist (\der z)_{su}dx_u
$$
for $\ell_1 \leq s\le t \leq \ell_2 $.   If $z$ admits the decomposition (\ref{f3.0.4})
we
obtain
\begin{align}
\ist (\der z)_{su}dx_u  =
\ist \left( \zeta_{s} \bx^{\1}_{su}  +\rho_{su} \right)  dx_u
=  \zeta_{s} \ist \bx^{\1}_{su}  \, dx_u + \ist \rho_{su}  \, dx_u.
\label{heuristic}
\end{align}
Moreover, if we set
$$\xd_{st}:= \ist \bx^{\1}_{su}  \, dx_u, \qquad \ell_1 \leq s\le t \leq \ell_2,$$
then it is quickly verified that $\xd$ is the  L\'evy area associated to $x$. Hence we can write
$$
\ist  z_udx_u
=  z_s \bx^{\1}_{sz} +  \zeta_{s}\,  \xd_{st}  + \ist \rho_{su}  \, dx_u.
$$
 Now recast this equation as
\begin{align}
  \ist \rho_{su}  \, dx_u = \ist  z_u \, dx_u  - z_s \bx^{\1}_{st} -  \zeta_{s}\,  \xd_{st},
\label{exp1:irhodx}
\end{align}
and apply  the increment operator  $\der$ to both  sides of this equation.
For smooth paths $z$ and $x$
we have
$$
\der \left (\int z \, dx \right)=0,
\qquad \qquad
\der(z \,\bx^{\1})= - \der z \, \bx^{\1},
$$
by Proposition \ref{difrul} (recall also our convention (\ref{cvpdt}) on products of increments).
Hence applying these relations to the right hand side of
(\ref{exp1:irhodx}), using the decomposition (\ref{f3.0.4}), the properties of the L\'evy area and
again Proposition \ref{difrul},
we obtain
\begin{align*}
 \left[\der \left( \int \rho \, dx \right) \right]_{sut}  &= \,
\der z_{su} \bx^{\1}_{ut}
+ \der\zeta_{su} \, \xd_{ut}
-\zeta_{s}  \, \delta \xd _{sut} \\
&= \zeta_s   \bx^{\1}_{su}   \, \bx^{\1}_{ut}
+\rho_{su} \, \bx^{\1}_{ut}
+ \der\zeta_{su} \xd_{ut}
-\zeta_{s}  \bx^{\1}_{su}   \, \bx^{\1}_{ut}
\\
&=\rho_{su} \bx^{\1}_{ut}+  \der\zeta _{su} \, \xd_{ut}.
\end{align*}
Summarizing, we have derived  the representation
$$
\left[\der \left( \int \rho \, dx \right) \right]_{sut}  =\rho_{su} \bx^{\1}_{ut}+  \der\zeta _{su} \,
\xd_{ut}.
$$
As we are dealing with smooth paths we have
$
\der \left( \int \rho \, dx \right)$ lies into the space $\mathcal{Z}\cac_3^{1+}$ and thus  belongs to the domain
 of $\laa$   due  to
Proposition \ref{prop:Lambda}.  (Recall that $\der\der=0$.) Hence, it follows
$$
\ist \rho_{su} \, dx_u =
\laa_{st}\lp \rho\, \bx^{\1}  +
\der\zeta \, \xd \rp,
$$ and
 inserting this identity into (\ref{heuristic}) we end up with
$$
\ist  z_udx_u
=  z_s \, \bx^{\1}_{st} +  \zeta_{s}\,  \xd_{st}  +
\laa_{st}\lp \rho\, \bx^{\1}  +
\der\zeta \, \xd \rp.
$$
Since in addition
$$     \rho\, \bx^{\1}  +
\der \zeta \, \xd =   - \der(z \, \bx^{\1} + \zeta \, \xd)  ,$$
we can also write this as
$$ \int  z_udx_u = (\id-\Lambda \delta) (z \, \bx^{\1} + \zeta \, \xd).$$
Thus we have expressed the Riemann-Stieltjes integral of $z$ with respect to $x$
in  terms of the sewing map $\laa$, the couple $(\mathbf{x}^{\mathbf{1}},\mathbf{x}^{\mathbf{2}})$
 and of increments of $z$. This can now be generalized to the non-smooth case.
Note that Corollary \ref{cor:integration} justifies the use of the notion of integral.

\begin{proposition}\label{intg:mdx}
For fixed $\frac13 < \ka \leq \ga$,
let $x$ be a path satisfying Hypothesis \ref{h1} on an arbitrary interval $\ott$. Furthermore,  let
$z \in\cq^x_{\ka,\alpha}([\ell_1,\ell_2];\R^{d})$ such that the
increments of $z$ are given by (\ref{f3.0.4}).
Define $\hat{z}$ by $\hat{z}_{\ell_1}= \hat{\alpha}$ with $\hat{\alpha} \in \R$  and
\begin{equation}\label{dcp:mdx}
\der \hat{z}_{st}= \lc (\operatorname{id}-\Lambda \delta) (z^{i} \bx^{\1,i} + \zeta^{ji}   \bx^{\2,ij}) \rc_{st}
\end{equation}
for $\ell_1 \leq s \leq t \leq \ell_2$.
Then $\cj(z^*\, dx):= \hat{z} $ is a well-defined element of $\cq^x_{\ka, \hat{\alpha}}([\ell_1,\ell_2];\R)$ and coincides with the usual Riemann integral, whenever
$z$ and $x$ are smooth functions.
\end{proposition}

Moreover, the H\"older norm of $\cj(z^*\, dx)$ can be estimated in terms of the H\"older norm of the integrator $z$. (For this and also for a proof of the above Proposition, see e.g.  \cite{Gu}.) This allows to use a fixed point argument to obtain the existence of a unique solution for rough differential equations.

\begin{theorem}\label{thm:Lip}
For fixed $\frac13 < \ka < \ga$, let $x$ be a path satisfying Hypothesis \ref{h1} on an arbitrary interval $\ott$. Consider a given initial condition $a$ in $\R^m$ and $(V_1,\ldots,V_d)$ a family of $C^3$ vector fields of $\R^m$, bounded with bounded derivatives. Let $\|f\|_{\mu,\infty}=\|f\|_{\infty}+\| \der f \|_{\mu}$ be the usual H\"older norm of a path
$f \in \cac_{1}([0,T]; \R^l)$. Then we have:
\begin{enumerate}
\item
Equation (\ref{de}) admits a unique solution $y$  in
$\cq^x_{\ka,a}([0,T];\R^{m})$ for  any $T>0$, and there exists a polynomial $P_T: \R^2 \rightarrow \R^+$ such that
\begin{equation}\label{control-sol-y}
\cn[y;\cq^x_{\ka,a}([0,T];\R^{m})] \leq P_T(\| \bx^{\1}\|_{\gamma},\|  \xd \|_{2 \gamma})
\end{equation}
holds.
\item
 Let
$F: \R^{m} \times \cac_1^{\ga}([0,T];\R^{d})\times
\cac_2^{2\ga}([0,T];\R^{m, m}) \rightarrow
  \cac_1^{\ga}([0,T];\R^{m})$  be the mapping defined by $$ F\left( a ,\bx^{\1},\xd  \right) =y, $$
where $y$ is the unique solution of equation (\ref{de}).  This mapping
is locally
 Lipschitz continuous in the following sense:
Let $\tilde{x} $ be another driving rough path with corresponding
 L\'evy area $\xdt $ and $\tilde{a}$ be another initial condition. Moreover denote by
$ \tilde{y}$ the unique solution of the corresponding differential equation.
        Then, there exists an increasing function $K_T: \R^4 \rightarrow \R^+$
such that
\begin{align}\label{lipsch-cont}
  \qquad \,\,\,  \| y - \tilde{y}\|_{\ga,\infty,T}
  &\leq K_T(\| \bx^{\1}\|_{\gamma} ,  \| \tilde\bx^{\1} \|_{\gamma},  \|  \xd \|_{2 \gamma}, \|  \xdt \|_{2 \gamma}  ) \, \\& \qquad \qquad \times    \left(  |a - \tilde{a}| +  \| \bx^{\1} - \tilde\bx^{\1}\|_{\gamma} +
\|  \xd- \xdt \|_{2 \gamma}  \right) \nonumber
\end{align}
holds.
\end{enumerate}
\end{theorem}

The theorem above is borrowed from \cite{FV,Gu,LQ}, and we send the reader to these references for more details on the topic.

\subsection{Fractional Brownian motion}
\label{sec:fbm}
We shall recall here how the abstract Theorem \ref{thm:Lip} applies to fractional Brownian motion. We will also give some basic notions on stochastic analysis with respect to fBm, mainly borrowed from \cite{N-bk}, which will turn out to be useful in the sequel.

\smallskip

As already mentioned in the introduction, on a finite interval $\ott$ and for some fixed $H\in(1/3,1/2)$, we consider $(\oom,\cf,\bp)$ the canonical probability space associated with fractional
Brownian motion with Hurst parameter $H$. That is,  $\oom=\cac_0(\ott;\R^d)$ is the Banach space of continuous functions
vanishing at $0$ equipped with the supremum norm, $\cf$ is the Borel sigma-algebra and $\bp$ is the unique probability
measure on $\oom$ such that the canonical process $B=\{B_t, \; t\in [0,T]\}$ is a $d$-dimensional fractional Brownian motion with Hurst
parameter $H$. Specifically, $B$ has $d$ independent coordinates, each one being a centered Gaussian process with covariance given by (\ref{eq:def-cov-fbm}).

\subsubsection{Functional spaces}\label{sec:fct-spaces}
Let $\ce$ be the set of the space of $d$-dimensional elementary functions on $\ott$:
\begin{multline}\label{eq:def-elem-fct}
\ce=\Big\{ f=(f_1,\ldots,f_d);\,\,f_j=\sum_{i=0}^{n_j-1} a_i^j
\1_{[t_i^j, t_{i+1}^j)}\,, \quad
0=t_0<t_1^j<\cdots<t_{n_j-1}^j<t_{n_j}^j=T,\\
\text{ for }j=1,\ldots,d\Big\}\,.
\end{multline}
We call $\ch$ the completion of $\ce$ with respect to the semi-inner product
\[
\lla f,\, g\rra_{\ch}=\sum_{i=1}^{d} \lla f_{i},\, g_{i}\rra_{\ch_{i}},
\quad\mbox{where}\quad
\langle \1_{[0,t]}, \1_{[0,s]} \rangle_{\ch_{i}} := R(s,t), \quad s,t \in [0,T].
\]
Then, one constructs an isometry $K^*_H: \ch \rightarrow  L^2([0,1];\R^d)$  such that
$$
K^*_H\lp \1_{[0,t_{1}]},\ldots,\1_{[0,t_{d}]}\rp
= \lp \1_{[0,t_1]} K_H(t_1,\cdot),\ldots, \1_{[0,t_d]} K_H(t_d,\cdot)\rp,
$$
where the kernel $K_H$ is given by
\begin{multline}\label{eq:def-K}
K_H(t,u)= c_H \Big[ \lp  \frac{u}{t}\rp^{\frac 12-H} (t-u)^{H-\frac 12}  \\
+\left(\frac 12-H\right) \, u^{\frac 12-H} \int_u^t v^{H-\frac 32} (v-u)^{H-\frac 12} \, dv \Big] \1_{\{0<u<t\}},
\end{multline}
with  a strictly positive constant $c_H$, whose exact value is irrelevant for our purpose. Notice that this kernel verifies $R_H(t,s)= \int_0^{s\land t} K_H(t,r) K_H(s,r)\, dr$. Moreover, observe that $K^*_H$ can be represented
in the following form: for $\vp=(\vp_1,\ldots,\vp_d)\in\ch$, we have $K^*_H \vp=(  K^*_H \vp^1,\ldots,K^*_H\vp^d )$, with
\[
K^*_H \vp=\lp  K^*_H \vp^1,\ldots,K^*_H\vp^d \rp,
\quad\mbox{where}\quad
[K^*_H \vp^i]_t = d_H \, t^{1/2-H} \lc  D_{T^-}^{1/2-H} \lp u^{-(1/2-H)} \vp^{i}\rp \rc_t,
\]
for a strictly positive constant $d_H$. In particular, each $\ch_i$ is a fractional integral space of the form $\ci_{T^-}^{1/2-H}(L^2([0,T]))$ and $\cac_1^{1/2-H}([0,T])\subset \ch_i$.

\subsubsection{Malliavin derivatives}\label{sec:malliavin-derivatives}

Let us start by defining the Wiener integral with respect to $B$: for any element $f$ in $\ce$ whose expression is given as in (\ref{eq:def-elem-fct}), we define the Wiener integral of $f$ with respect to $B$ as
\[
B(f):=\sum_{j=1}^d\sum_{i=0}^{n_j-1} a_i^j (B_{t_{i+1}^j}^{j}
-B_{t_i^j}^{j})\,.
\]
We also denote this integral as $ \int_0^T f(t) dB_t$, since it coincides with a pathwise integral with respect to $B$.

\smallskip

For $\theta:\R\rightarrow \R$, and $j\in\{1,\ldots,d\}$, denote by
$\theta^{[j]}$ the function with values in $\R^d$ having all the
coordinates equal to zero except the $j$-th coordinate that equals
to $\theta$. It is readily seen that
$$
\be\lc B\lp \1_{[0,s)}^{[j]}\rp \, B\lp \1_{[0,t)}^{[k]}\rp \rc
=\delta_{j,k}R_{s,t}.
$$
This definition can be extended by linearity and closure to elements of $\ch$, and we obtain the relation
$$
\be\lc B(f) \, B(g)\rc =\langle f,g\rangle_{\ch},
$$
valid for any couple $(f,g)\in\ch^2$. In particular, $B(\cdot)$ defines an isometric map from $\ch$  into a subspace of $L^2(\Omega)$. It should be pointed out that $(\oom,\ch,\bp)$ defines an abstract Wiener space, on which chaos decompositions can be settled. We do not develop this aspect of the theory for sake of conciseness, but we will use later the fact that all $L^p$ norms are equivalent on finite chaos.

\smallskip

We can now proceed to the definition of Malliavin derivatives, for which we need an additional notation:
\begin{notation}
For $n,p\ge 1$, a function $f\in\cac^{p}(\R^{n};\R)$ and any tuple $(i_1,\ldots i_p)\in\{1,\ldots,d\}^{p}$, we set $\partial_{i_1\ldots i_p} f$ for $\frac{\partial^{p} f}{\partial x_i\ldots \partial x_p}$.
\end{notation}
With this notation in hand, let $\cs$ be the
family of smooth functionals $F$ of the form
\begin{equation}\label{eq:def-smooth-fct}
F=f(B(h_1),\dots,B(h_n)),
\end{equation}
where $h_1,\dots,h_n\in \ch$, $n\geq 1$, and $f$ is a smooth function with polynomial growth, together with all its derivatives. Then, the Malliavin derivative of such a functional $F$ is the $\ch$-valued random variable defined by
\[
D F= \sum_{i=1}^n \partial_{i} f(B(h_1),\dots,B(h_n)) \, h_i.
\]
For all $p>1$, it is known that the operator $D$ is closable from $L^p(\oom)$ into $L^p(\oom; \ch)$ (see e.g. \cite[Chapter 1]{N-bk}).
We will still denote by $D$ the closure of this operator, whose domain is usually denoted by $\D^{1,p}$ and is defined
as the completion of $\cs$ with respect to the norm
\[
\|F\|_{1,p}:= \left( E(|F|^p) + E( \|D F\|_\ch^p ) \right)^{\frac 1p}.
\]
It should also be noticed that partial Malliavin derivatives with respect to each component $B^{j}$ of $B$ will be invoked: they are defined, for a functional $F$ of the form (\ref{eq:def-smooth-fct}) and $j=1,\dots,d$, as
\begin{equation*}
D^j F=\sum_{i=1}^n  \partial_{i} f(B(h_1),\dots,B(h_n)) h_i^{[j]},
\end{equation*}
and then extended by closure arguments again. We refer to \cite[Chapter 1]{N-bk} again for the definition of higher derivatives and Sobolev spaces $\D^{k,p}$ for $k>1$.

\subsubsection{Levy area of fBm}

There are many ways to define the Levy area $\bb^{\2}$ associated to fBm, and the reader is referred to \cite[Chapter 15]{FV} for a complete review of these. The recent paper \cite{FU} is however of special interest for us, since it enables a direct definition of $\bb^{\2}$ by Wiener chaos techniques. It can be summarized in the following way:
\begin{proposition}\label{prop:levy-fbm}
Let $1/3<H<1/2$ be a fixed Hurst parameter. Then the fBm $B$ belongs almost surely to any space $\cac_1^{\ga}$ for $\ga<H$, and it gives raise to an increment $\bb^{\2}\in\cac_2^{2\ga}$ which satisfies Hypothesis \ref{h1}. Furthermore, for any $0\le s<t\le T$, $\bb^{\2}_{st}$ is an element of the second chaos associated to $B$, and
\begin{equation*}
\be\lc  \lln\bb^{\2}_{st}\rrn^{p}\rc \le c_p \, (t-s)^{2Hp}, \quad p\ge 1.
\end{equation*}
\end{proposition}
Moreover, the iterated integrals of $B$ can be obtained as limits of Riemann type integrals. Indeed, for $k\ge 1$ and $0\le s<t\le T$, consider the simplex
\begin{equation}\label{eq:def-simplex}
\cs_{k}([s,t])
=\lcl (u_1,\dots, u_k); \,  s\le u_1 < \cdots < u_k \le t\rcl.
\end{equation}
For a given partition $\Pi$ of $\ott$, we also denote by $B^{\Pi}$ the linearization of $B$ based on $\Pi$. Combining the results of \cite{FU,FV}, the following proposition holds true:
\begin{proposition}\label{prop:limit-iterated-intg}
Let $k\ge 1$, and for a sequence of partitions $(\Pi_n)_{n\ge 1}$, set $B^{n}:=B^{\Pi_n}$. For $0\le s<t\le T$ and $(i_1,\ldots,i_k)\in\{1,\ldots,d\}^{k}$, we consider then
\begin{equation*}
\bb^{\bk,n,i_1,\ldots,i_k}_{st}=\int_{\cs_{k}([s,t])} dB_{u_1}^{n,i_1} \cdots dB_{u_k}^{n,i_k},
\end{equation*}
understood in the Riemann sense. Then there exists a sequence of partitions $(\Pi_n)_{n\ge 1}$ such that $\bb^{\bk,n,i_1,\ldots,i_k}$ converges almost surely and in $L^2$, as an element of $\cac_2^{k\ga}$ for any $\ga<H$, to an element called  $\bb^{\bk,i_1,\ldots,i_k}$. When $k=1$, we obtain the increment $\der B$ of our fBm. When $k=2$, the limit corresponds to the increment $\bb^{\2}$ of Proposition \ref{prop:levy-fbm}.
\end{proposition}

As a corollary of the previous considerations, we have the
\begin{proposition}
Assume $1/3<H<1/2$. Then Theorem \ref{thm:Lip} applies almost surely to the fBm paths, enhanced with the Levy area $\bb^\2$.
We are thus able to solve equation
\begin{equation}\label{eq:fractional-sde}
dy_t=\sum_{i=1}^{d}V_i(y_t) \, dB_t^{i},
\quad t \in \lc 0, T \rc, \qquad  y_0=a,
\end{equation}
under the conditions of Theorem \ref{thm:Lip}.
\end{proposition}

\section{A Norris type lemma}
\label{sec:norris-lemma}

Norris' lemma \cite{No} is one of the basic ingredients  in order to obtain smoothness of densities for solutions to stochastic differential equations under hypoelliptic conditions, and was already extended to fBm with Hurst parameter $H>1/2$ in \cite{BH}. We shall extend  in the current section  this lemma to the rough paths context. A preliminary step  along  this direction consists  of   proving the following elementary lemma:

\begin{lemma}
Let $0<\al<\rho<1$, and consider $b\in\cac([0,T])$. Then for any $0<\eta<1$, we have
\begin{equation}\label{eq:interpol-ineq}
\|b\|_{\al,\infty} \le C_{\al, \rho} \left[\eta \, \|b\|_{\rho,\infty}+\eta^{-1/(\rho-\al)}\|b\|_{L_1}\right]\,,  
\end{equation}
where we recall that  $\|b\|_{\al,\infty}$ has been defined by (\ref{eq:def-holder-sup}).
\end{lemma}

\begin{proof}  Recall that $\|b\|_\al$ has been defined by (\ref{eq:def-holder-C2}). Thus, for $0<\al<\rho<1$,
we have
\begin{equation*}
\|b\|_\al =
\sup_{s,t} \lc \left(\frac{|\der b_{st}|}{|t-s|^{\rho}}\right)^{\frac{\al}{\rho}}
|\der b_{st}|^{1-\frac{\al}{\rho}} \rc
\le 2^{1-\frac{\al}{\rho}} \|b\|_\infty^{1-\frac{\al}{\rho}} \|b\|_\rho^{\frac{\al}{\rho}}.
\end{equation*}
Thus, for an arbitrary constant $ \eta>0$, we have
\begin{equation*}
\|b\|_\al \le
C_{\al, \rho} \left( \eta \|b\|_\rho +\eta ^{-\frac{\al}{\rho-\al}}\|b\|_\infty\right)
\le C_{\al, \rho} \left( \eta \|b\|_\rho +\eta ^{-\frac{\al}{\rho-\al}}\|b\|_\infty\right),
\end{equation*}
thanks to Young's inequality. Therefore for an arbitrary constant $0<\eta\le 1$
\begin{equation*}
 \|b\|_\al
\le C_{\al, \rho} \left( \eta \|b\|_\rho +\eta ^{-\frac{\al}{\rho-\al}}\|b\|_\infty\right)\,. 
\end{equation*}
Invoke now the interpolation inequality \cite[formula (3.17)]{BH} in order to get
\begin{eqnarray*}
\|b\|_\al
&\le&
C_{\al, \rho} \left( \eta \|b\|_\rho +\eta ^{-\frac{\al}{\rho-\al}}
\left[ \ga \|b\|_{\rho} +\ga^{-\frac1{1+\al}} \|b\|_{L_1}\right] \right)\\
&=&C_{\al, \rho} \left[\eta \|b\|_\rho+\eta^{-1/(\rho-\al)}\|b\|_{L_1}\right]\,,
\end{eqnarray*}
where we have chosen $\ga=\eta ^{\frac{\al}{\rho-\al}}$. Invoking again \cite[formula (3.17)]{BH} in order to go from $\|\cdot\|_{\al}$ to $\|\cdot\|_{\al,\infty}$ norms
and  this finishes our proof.

\end{proof}

\smallskip

We can now turn to the announced Norris type lemma, whose proof is an adaptation of \cite{BH} to the case of controlled processes. 
\begin{proposition}\label{prop:norris-lemma}
Assume $B$ is a fractional Brownian motion with $H>1/3$. Let $z$ be
a controlled path in $\cq_{\ga}^{B}(\R^{m})$, with decomposition
\begin{equation}\label{eq:dcp-z-norris}
\der z_{st}^{i}=\zeta_s^{i,j} \, \bb^{\1,j}_{st} + r_{st}^{i}.
\end{equation}
We assume that $1/3<\al<\ga<H$ and that the quantity $\be[\cn^{p}[z;\, \cq_{\ga}^{B}(\R^{m})]]$ is finite for all $p\ge 1$.
Set $\der y_{st}=\cj(z^* dB)$ according to Proposition~\ref{intg:mdx}. Then there exists $q>0$ such that, for every $p>0$ we can find a strictly positive constant $c_p$ such that
\begin{equation*}
\bp\lp  \|y\|_{\ga,\infty} < \ep, \mbox{ and }
\|z\|_{\al,\infty} > \ep^q \rp < c_p \, \ep^p.
\end{equation*}
\end{proposition}

\begin{proof}
In order to avoid cumbersome indices, we shall prove our result in the
case of 1-dimensional processes. Generalization to the multidimensional setting is
a matter of trivial considerations. We also work on the interval $\ou$ instead of
$\ott$ for the  sake of notational simplicity. As a last preliminary observation, note that if $z$ admits
the decomposition (\ref{eq:dcp-z-norris}), then according to (\ref{dcp:mdx}) we have
\begin{equation*}
\der y= z \, \bb^{\1} + \zeta \, \bb^{\2} + y^{\sharp},
\quad\mbox{where}\quad
y^{\sharp}=\laa\lp r \, \bb^{\1} + \der\zeta \, \bb^{\2} \rp.
\end{equation*}

\smallskip

Similarly to what is done in \cite{BH}, we consider two time scales $\delta\ll \Delta\ll 1$. We assume moreover that $\De/\de=r$ with $r\in\N$. We use a partition $\{t_n;\, n\le 1/\delta\}$ of $[0,1]$ with $t_n=n\de$, so that $t_{Nr}=N\De$. Some increments below will then be frozen on the time scale $\De$, in order to take advantage of some averaging properties of the process $B$.

\smallskip

\noindent
\textit{Step1: Coarse graining on increments.}
Consider then $n$ such that $(N-1)r\le n\le Nr-1$, so that $(N-1)\De\le t_n\le N\De -1$. According to (\ref{dcp:mdx}), we have
\begin{eqnarray*}
\der y_{t_n t_{n+1}}&=&z_{t_n} \, \bb^{\1}_{t_n t_{n+1}}
+ \zeta_{t_n} \, \bb^{\2}_{t_n t_{n+1}}
+ y^{\sharp}_{t_n t_{n+1}}\\
&=&z_{N\De} \, \bb^{\1}_{t_n t_{n+1}} - \der z_{t_n,N\De}\, \bb^{\1}_{t_n t_{n+1}}
+ \zeta_{t_n} \, \bb^{\2}_{t_n t_{n+1}}
+ y^{\sharp}_{t_n t_{n+1}},
\end{eqnarray*}
where $y^{\sharp}$ is an increment in $\cac_2^{3\ga}$. Thus
\begin{equation}\label{eq:exp-prd-z-x1}
z_{N\De} \, \bb^{\1}_{t_n t_{n+1}}=
\der y_{t_n t_{n+1}}+\der z_{t_n,N\De}\, \bb^{\1}_{t_n t_{n+1}}
- \zeta_{t_n} \, \bb^{\2}_{t_n t_{n+1}}
- y^{\sharp}_{t_n t_{n+1}}.
\end{equation}
Set now $X_N=\sum_{n=(N-1)r}^{Nr-1} \| \bb^{\1}_{t_n t_{n+1}} \|^4$ and $Y_N=X_N^{1/4}$. Then
\begin{equation*}
|z_{N\De}|^{4} |X_N|
= \sum_{n=(N-1)r}^{Nr-1} \lln z_{N\De} \, \bb^{\1}_{t_n t_{n+1}} \rrn^4.
\end{equation*}
Furthermore, invoking relation (\ref{eq:exp-prd-z-x1}), we get
\begin{equation*}
|z_{N\De} \, \bb^{\1}_{t_n t_{n+1}}| \le
 \|y\|_{\ga} \de^{\ga} + \|z\|_{\ga} \|B\|_{\ga} \de^{\ga} \De^{\ga}
+ \|\zeta\|_{\infty} \|\bb^{\2}\|_{2\ga} \de^{2\ga} + \|y^{\sharp}\|_{3\ga} \de^{3\ga}.
\end{equation*}
Raising this inequality to power 4 and summing over $(N-1)r\le n\le Nr-1$, we obtain
\begin{equation*}
|z_{N\De}^{4} \, X_N| \le
\de^{4\ga-1} \De
\lp \|y\|_{\ga}  + \|z\|_{\ga} \|B\|_{\ga} \De^{\ga}
+ \|\zeta\|_{\infty} \|\bb^{\2}\|_{2\ga} \de^{\ga} + \|y^{\sharp}\|_{3\ga} \de^{2\ga}
\rp^{4},
\end{equation*}
and therefore
\begin{equation*}
|z_{N\De}| Y_N \le  \de^{\ga-1/4} \De^{1/4}
\lp
\|y\|_{\ga}  + \|z\|_{\ga} \|B\|_{\ga} \De^{\ga}
+ \|\zeta\|_{\infty} \|\bb^{\2}\|_{2\ga} \de^{\ga} + \|y^{\sharp}\|_{3\ga} \de^{2\ga}
\rp.
\end{equation*}
Sum now over $N$ (recall that $1\le N\le 1/\De$) in order to get
\begin{equation}\label{eq:zN-YN}
\sum_{N=1}^{1/\De}|z_{N\De}| Y_N
\le  \de^{\ga-1/4} \De^{-3/4}
\lp
\|y\|_{\ga}  + \|z\|_{\ga} \|B\|_{\ga} \De^{\ga}
+ \|\zeta\|_{\infty} \|\bb^{\2}\|_{2\ga} \de^{\ga} + \|y^{\sharp}\|_{3\ga} \de^{2\ga}
\rp.
\end{equation}

\smallskip

\noindent
\textit{Step 2:  Behavior of a 4\textsuperscript{th} order variation.}
Throughout the proof, we shall use the notations $\lesssim,\asymp$ given in the introduction.
For $K\ge 1$, set  $\tilde X_K=\sum_{n=1}^{K} | \bb^{\1}_{t_n t_{n+1}} |^4$. We shall prove that
\begin{equation}\label{eq:bnd-mom-tilde-XK}
\be[\tilde X_K] \asymp K\de^{4H},
\quad\mbox{and}\quad
\var(\tilde X_K)\lesssim K\de^{8H}.
\end{equation}

\smallskip

Indeed, a simple scaling argument shows that $\tilde X_K\stackrel{(\cl)}{=}\de^{4H}\, \hat X_K$, with $\hat X_K=\sum_{n=1}^{K} | \bb^{\1}_{n,n+1} |^4$. Introduce now the k-th Hermite polynomial $H_k$ (see \cite{N-bk} for a definition and properties of these objects), and notice in particular that $H_2(x)=x^2-1$ and $H_4(x)=x^4-6x^2+3$. This enables to decompose $\hat X_K$ as 
\begin{equation}\label{eq:dcp-hat-XK}
\hat X_K= \sum_{n=1}^{K} \lc H_4(\bb^{\1}_{n,n+1}) +  6 H_2(\bb^{\1}_{n,n+1}) \rc + 3K.
\end{equation}

\smallskip

Recall now that for a centered Gaussian vector $(U,V)$ in $\R^2$ such that $\be[U^2]=\be[V^2]=1$ we have
\begin{equation}\label{eq:covariance-hermite}
\be[H_k(U)]=0, 
\quad\mbox{and}\quad
\be \lc H_k(U) \, H_l(V) \rc = k! \lp  \be[U\, V]\rp^k \1_{(k=l)}.
\end{equation}
Plugging this identity in \eref{eq:dcp-hat-XK}, this immediately yields $\be[\hat X_K]=3K$, which is our first assertion in \eref{eq:bnd-mom-tilde-XK}. In addition, the second part of \eref{eq:covariance-hermite} entails
\begin{equation*}
\var(\hat X_K)=2 \, \sum_{n_1,n_2=1}^{K} \lp 12 \, \al_{n_1,n_2}^4 + \al_{n_1,n_2}^2 \rp:= 2S_K,
\end{equation*}
where we have set 
\begin{equation*}
\al_{n_1,n_2}=\be[\bb^{\1}_{n_1,n_1+1} \, \bb^{\1}_{n_2,n_2+1}]
=\frac12 \lc |n_2-n_1+1|^{2H} + |n_2-n_1-1|^{2H} -2  |n_2-n_1|^{2H} \rc.
\end{equation*}
Summarizing, we have obtained that
\begin{equation}\label{eq:var-XK-2}
\var(\tilde X_K)= \de^{8H} \, \var(\hat X_K) = 2\,  \de^{8H} \, S_K.
\end{equation}
We will now prove that $S_K\lesssim K$. Indeed, write first
\begin{equation*}
S_K=\sum_{1\le n_1 \le K} \lp 12 \, \al_{n_1,n_1}^{4} +  \al_{n_1,n_1}^{2} \rp
+2 \sum_{1\le n_1<n_2 \le K} \lp 12 \,  \al_{n_1,n_2}^{4} +  \al_{n_1,n_2}^{2} \rp
:= S_K^{1}+2 S_K^{2}.
\end{equation*}
Then, since  $\al_{n_1,n_1}=1$, it is readily checked that $S_K^{1}\lesssim K$. Moreover, the term $S_K^{2}$ can be decomposed into
\begin{equation*}
S_K^{2}= \sum_{1\le n_1 \le K-1} \lp 12 \, \al_{n_1,n_1+1}^{4} +  \al_{n_1,n_1+1}^{2} \rp
+ \sum_{1\le n_1,n_2 \le K, \, n_2-n_1\ge 2} \lp 12 \, \al_{n_1,n_2}^{4} +  \al_{n_1,n_2}^{2} \rp
:= S_K^{21}+S_K^{22}.
\end{equation*}
Notice now that $\al_{n_1,n_1+1}=-[1-2^{-2(H-1/2)}]$, which easily yields $S_K^{21}\le c_H K$.

\smallskip

As far as $S_K^{22}$ is concerned, write for $n_2-n_1\ge 2$
\begin{equation*}
\al_{n_1,n_2}=H (2H-1) \int_0^1 dr \int_{-r}^{r}
\lln  (n_2-n_1)+u\rrn^{2H-2} \, du,
\end{equation*}
which immediately yields $\al_{n_1,n_2}\lesssim |(n_2-n_1)-1|^{2H-2}$. Thus
\begin{eqnarray*}
S_K^{22}&\lesssim&
\sum_{n_1=1}^{K-2}\sum_{n_2=n_1+2}^{K} |(n_2-n_1)-1|^{8H-8}+ 3 |(n_2-n_1)-1|^{4H-4} \\
&\lesssim&
\sum_{n_1=1}^{K-2} \sum_{m=1}^{K-n_1+1} m^{4H-4}
\leq K \sum_{m=1}^{\infty} m^{4H-4} \lesssim K,
\end{eqnarray*}
since $\sum_{m=1}^{\infty} m^{4H-4}$ is finite whenever $H<3/4$.

\smallskip

Gathering our bounds on $S_K^{1}, S_K^{21}$ and $S_K^{22}$, we obtain $S_K\lesssim K$, and plugging this bound into (\ref{eq:var-XK-2}), we end up with $\var(\tilde X_K)\lesssim K\,  \de^{8H}$, which is our claim. This finishes the proof of~(\ref{eq:bnd-mom-tilde-XK}).

\smallskip

\noindent
\textit{Step 3: Concentration inequalities for $Y_N$.}
Let us recall that $X_N$ is in the 4\textsuperscript{th} chaos of the fBm $B$. Hence, a result by Borell \cite{Bo} entails
\begin{equation*}
\bp\lp  \frac{|X_N-\be[X_N]|}{\lc\var(X_N)\rc^{1/2}} \ge u\rp
\le c_1 e^{-c_2 u^{1/2}},
\qquad u \ge 0,
\end{equation*}
for two universal constant $c_1,c_2>0$. With (\ref{eq:bnd-mom-tilde-XK}) in hand, this yields
\begin{equation}\label{eq:concentration-XN}
\bp\lp  |X_N-3\De \de^{4H-1}| \ge \De^{1/2} \de^{4H-1/2} u\rp
\le c_1 e^{-c_2 u^{1/2}},
\qquad u \ge 0.
\end{equation}

\smallskip

We now wish to produce a concentration inequality for $Y_N=X_N^{1/4}$. Since $X_N$ is a small random quantity of order $\De \de^{4H-1}$, let us use the inequality
$$
|b^{1/4}-a^{1/4}| \lesssim \xi^{-3/4} |b-a|,
\quad\mbox{where}\quad
\xi\in(a\wedge b,\, a\vee b).
$$
Apply this with $\xi=\frac32 \De \de^{4H-1}$ in order to get
\begin{equation}\label{eq:ineq-proba-YN}
\bp\lp  |Y_N-3^{1/4}\De^{1/4} \de^{H-1/4}| \ge c \, \De^{-3/4} \de^{-3(4H-1)/4}\De^{1/2} \de^{4H-1/2} u\rp
\le A_1+A_2,
\end{equation}
with
\begin{eqnarray*}
A_1&=&
\bp\lp  |X_N-3\De \de^{4H-1}| \ge c\, \De^{1/2} \de^{4H-1/2} u, \, X_N\ge \frac32 \De \de^{4H-1}\rp \\
A_2&=&
\bp\lp   X_N\le \frac32 \De \de^{4H-1}\rp.
\end{eqnarray*}
Furthermore, a straightforward application of (\ref{eq:concentration-XN}) gives
\begin{equation*}
A_1\le c_1 e^{-c_2 u^{1/2}},
\quad\mbox{and}\quad
A_2\le c_1 e^{-c_2 (\De/\de)^{1/4}}.
\end{equation*}
Plugging these inequalities into (\ref{eq:ineq-proba-YN}), we end up with the following concentration inequality for $Y_N$.
\begin{equation}\label{eq:concentration-YN}
\bp\lp  |Y_N-3^{1/4}\De^{1/4} \de^{H-1/4}| \ge c \, \De^{-1/4} \de^{H+1/4} u\rp
\le c_1 e^{-c_2 u^{1/2}} + c_1 e^{-c_2 (\De/\de)^{1/4}}.
\end{equation}
We shall thus retain the fact that $Y_N$ i a random quantity of order $3^{1/4}\De^{1/4} \de^{H-1/4}$, with fluctuations of order $\De^{-1/4} \de^{H+1/4}$:
\begin{equation}\label{eq:approx-YN}
\lln  Y_N-3^{1/4}\De^{1/4} \de^{H-1/4}\rrn \approx \de^{H+1/4} \De^{-1/4}.
\end{equation}

\smallskip

\noindent
\textit{Step 4: Use of the interpolation inequality.}
Start again from equation (\ref{eq:zN-YN}). One would like to have an approximation of the  $L^1$ norm of $z$ appearing on the left hand side of this inequality, that is one would like to replace $Y_N$ by $\De$. To this purpose, replace first $Y_N$ by its approximation $\De^{1/4} \de^{H-1/4}$ from the last step. This yields an inequality of the form:
\begin{multline*}
\De^{1/4} \de^{H-1/4} \sum_{N=1}^{1/\De}|z_{N\De}|
\le \|z\|_{\infty} \sum_{N=1}^{1/\De} \lln  Y_N- 3^{1/4}\De^{1/4} \de^{H-1/4}\rrn \\
+ \de^{\ga-1/4} \De^{-3/4}
\Big(
\|y\|_{\ga} + \|z\|_{\ga} \|B\|_{\ga}  \De^{\ga}
+ \|\zeta\|_{\infty} \|\bb^{\2}\|_{2\ga} \de^{\ga}
+ \|y^{\sharp}\|_{3\ga} \de^{2\ga}
\Big).
\end{multline*}
Rescale this inequality in order to get $\De$ multiplying on the left hand side, which gives:
\begin{multline}\label{eq:z-approx-L1}
\De  \sum_{N=1}^{1/\De}|z_{N\De}|  \\
\le \|z\|_{\infty} \, R_{\de,\De}
+ \de^{-(H-\ga)}
\Big(
\|y\|_{\ga} + \|z\|_{\ga} \|B\|_{\ga}  \De^{\ga}
+ \|\zeta\|_{\infty} \|\bb^{\2}\|_{2\ga} \de^{\ga}
+ \|y^{\sharp}\|_{3\ga} \de^{2\ga}
\Big),
\end{multline}
where
\begin{equation}\label{eq:exp-R-delta}
R_{\de,\De}:=\de^{-(H-1/4)} \De^{3/4}\sum_{N=1}^{1/\De} |  Y_N- 3^{1/4}\De^{1/4} \de^{H-1/4}|.
\end{equation}
Furthermore, it is well known that $|\De  \sum_{N=1}^{1/\De}|z_{N\De}| -\| z \|_{L^1}|\le \|z\|_{\ga} \De^{\ga}$, so that we can recast (\ref{eq:z-approx-L1}) into
\begin{multline}\label{eq:z-L1}
\| z \|_{L^1}
\le \|z\|_{\ga} \De^{\ga} + \|z\|_{\infty} \, R_{\de,\De} \\
+ \de^{-(H-\ga)}
\Big(
\|y\|_{\ga} + \|z\|_{\ga} \|B\|_{\ga}  \De^{\ga}
+ \|\zeta\|_{\infty} \|\bb^{\2}\|_{2\ga} \de^{\ga}
+ \|y^{\sharp}\|_{3\ga} \de^{2\ga}
\Big).
\end{multline}

\smallskip

Recall that we take $0<\al<\ga<H$ and set $\nu_H:=1/(\ga-\al)$. According to (\ref{eq:interpol-ineq}) we have
\begin{equation*}
\|z\|_{\al,\infty} \lesssim
\eta \|z\|_{\ga,\infty}+\eta^{-\nu_H}\|z\|_{L_1},
\end{equation*}
for any (small enough) constant $\eta$, and plugging (\ref{eq:z-L1}) into this last relation we obtain
\begin{multline*}
\|z\|_{\al,\infty} \lesssim
\eta^{-\nu_H} \De^{\ga} \|z\|_{\ga,\infty}  + \eta^{-\nu_H} \|z\|_{\infty} \, R_{\de,\De}
+\eta \|z\|_{\ga,\infty} \\
+\eta^{-\nu_H}
\Big(
\|y\|_{\ga,\infty} \de^{-(H-\ga)} + \|z\|_{\ga,\infty} \|B\|_{\ga} \de^{-(H-\ga)} \De^{\ga}
+ \|\zeta\|_{\infty} \|\bb^{\2}\|_{2\ga} \de^{2\ga-H}
+ \|y^{\sharp}\|_{3\ga} \de^{3\ga-H}
\Big).
\end{multline*}
Defining $\he$ as $\he=H-\ga$, we get
\begin{multline}\label{eq:ineq-z-al-1}
\|z\|_{\al,\infty} \lesssim
\eta^{-\nu_H} \De^{\ga} \|z\|_{\ga,\infty}  + \eta^{-\nu_H} \|z\|_{\infty} \, R_{\de,\De}
+\eta \|z\|_{\ga,\infty} \\
+\eta^{-\nu_H}
\Big(
\|y\|_{\ga,\infty} \de^{-\he} + \|z\|_{\ga,\infty} \|B\|_{\ga} \de^{-\he} \De^{\ga}
+ \|\zeta\|_{\infty} \|\bb^{\2}\|_{2\ga} \de^{\ga-\he}
+ \|y^{\sharp}\|_{3\ga} \de^{2 \ga-\he}
\Big).
\end{multline}

\smallskip

\noindent
\textit{Step 4: Tuning the parameters.}
Recall that we have chosen $0\ll\de\ll\De\ll 1$. We express this fact in terms of powers of $\ep$, by taking $\de=\ep^{\mu}$ and $\De=\ep^{\la}$, with $\mu>\la>0$. We also choose $\eta$ of the form $\eta=\ep^{\tau/\nu_H}$. We shall now see how to choose $\la,\mu,\tau$ conveniently: write (\ref{eq:ineq-z-al-1}) as
\begin{multline}\label{eq:ineq-z-al-2}
\|z\|_{\al,\infty} \lesssim
\varepsilon^{-\tau+\la\ga}  \|z\|_{\ga,\infty}  + \varepsilon^{-\tau} \|z\|_{\infty} \, R_{\de,\De}
+\varepsilon^{\tau/\nu_H} \|z\|_{\ga,\infty}
+
\|y\|_{\ga,\infty} \ep^{-\tau-\he\mu} \\
+ \|z\|_{\ga,\infty} \|B\|_{\ga} \ep^{-\tau-\he\mu+\la\ga}
+ \|\zeta\|_{\infty} \|\bb^{\2}\|_{2\ga} \ep^{-\tau+\mu(\ga-\he)}
+ \|y^{\sharp}\|_{3\ga}  \ep^{-\tau+\mu(2\ga-\he)}.
\end{multline}
In order to be able to bound $z$ when $y$ is assumed to satisfy $\|y\|_{\ga,\infty}\le\ep$, the coefficients in the right hand side of (\ref{eq:ineq-z-al-2}) should fulfill the following conditions:
\begin{itemize}
\item
The coefficient in front of $\|y\|_{\ga,\infty}$ should be smaller than $\ep^{-1}$.
\item
The other coefficients should be $\ll1$.
\end{itemize}
Looking at the exponents in (\ref{eq:ineq-z-al-2}), assuming that $\he$ is arbitrarily small  and letting for the moment $R_{\de,\De}$ apart, this imposes the following relations:
\begin{equation}\label{eq:cdt-lambda-gamma-tau}
\la\ga>\tau,\quad\mbox{and}\quad 0<\tau<1.
\end{equation}

\smallskip

Let us go back now to the evaluation of $R_{\de,\De}$,  
given by expression~(\ref{eq:exp-R-delta}), with the order of magnitude of  $| Y_N- \De^{1/4} \de^{H-1/4}|$ given by (\ref{eq:approx-YN}). Therefore
\begin{equation*}
R_{\de,\De}\approx \de^{-(H-1/4)} \De^{3/4} \De^{-1} \de^{H+1/8} \De^{-1/4}
=\de^{1/2} \De^{-1/2}.
\end{equation*}
Expressing this in terms of powers of $\ep$, we end up with
\begin{equation*}
\eta^{-\nu_H}  \, R_{\de,\De}\approx \ep^{\ka},
\quad\mbox{with}\quad
\ka=\frac{\mu-\la}{2} -\tau.
\end{equation*}
If we wish this remainder term to be small, this adds the condition
\begin{equation}\label{eq:cdt-mu-la-tau}
\mu-\la>2\tau,
\end{equation}
which can be fulfilled easily. From now on, we shall assume that both (\ref{eq:cdt-lambda-gamma-tau}) and (\ref{eq:cdt-mu-la-tau}) are satisfied.

\smallskip

\noindent
\textit{Step 5: Conclusion.}
Recall that we wish to study the probability $\bp(  \|y\|_{\ga,\infty} < \ep,$ and $\|z\|_{\al,\infty}>   \varepsilon^q )$. This quantity is obviously bounded by $B_1+B_2$, where
\begin{eqnarray*}
B_1&=&
\bp\lp  \|y\|_{\ga,\infty} < \ep, \,
\|z\|_{\al,\infty} > \ep^q, \,
|R_{\de,\De}|\le c\, \De^{-1/4} \de^{H+1/4} (\De/\de)^{\tilde\ep}
\rp \\
B_2&=&
\bp\lp
|R_{\de,\De}|\ge c\,\De^{-1/4} \de^{H+1/4} (\De/\de)^{\tilde\ep}
\rp,
\end{eqnarray*}
where $\tilde\ep$ is an arbitrary small positive constant. Furthermore, inequalities (\ref{eq:concentration-YN}) and (\ref{eq:exp-R-delta}) yield, for any $p\ge 1$,
\begin{equation*}
B_2\le c_1 e^{-c_2 (\De/\de)^{\tilde\ep/2}} + c_1 e^{-c_2 (\De/\de)^{1/4}}
\le c_{p,\la,\mu,\tilde\ep} \, \ep^p,
\end{equation*}
where we have used the fact that $\de/\De=\ep^{\mu-\la}$.

\smallskip

We can now bound $B_1$: notice that according to (\ref{eq:ineq-z-al-2}), if we are working on
$$
\Big( |R_{\de,\De}|\le c\, \De^{-1/4} \de^{H+1/4} (\De/\de)^{\tilde\ep} \Big)
\cap
\Big( \|y\|_{\ga,\infty} < \ep  \Big),
$$
then there exists a $\rho>0$ such that
\begin{equation*}
\|z\|_{\al,\infty} \lesssim \ep^{\rho}
\lc 1+ \cn^{2}[z;\cq_\ga^{B}(\R^m)] +\|B\|_{\ga}^{2} + \|\bb^{\2}\|_{2\ga}^{2}\rc.
\end{equation*}
Moreover, recall that $\cn[z;\cq_\ga^{B}(\R^m)]$ is assumed to be an  $L^r$ random variable for all $r\ge 1$, while $\|B\|_{\ga}$ and $\|\bb^{\2}\|_{2\ga}$ are also elements of $L^r$, since they can be bounded by a finite chaos random variable. Thus Tchebychev inequality can be applied here, which entails
$$
B_1\lesssim
\lp  1+\be\lc \cn^{2l}[z;\cq_\ga^{B}(\R^m)] +\|B\|_{\ga}^{2l} + \|\bb^{\2}\|_{2\ga}^{2l}\rc \rp
\ep^{l(\rho-q)},
$$
for an arbitrary $l\ge 1$. It is now sufficient to choose $q<\rho$ and $l$ large enough so that $l(\rho-q)=p$ to conclude the proof, by putting together our bounds on $B_1$ and $B_2$.

\end{proof}

\section{Malliavin calculus for solutions to fractional SDEs}
\label{sec:malliavin-calculus}
This section is the core of our paper, where we derive smoothness of density
for the solution to (\ref{eq:fractional-sde}). We first recall some classical
notions on representations of solutions to SDEs, and then move to
Malliavin calculus considerations.

\subsection{Representation of solutions to SDEs}

The first representations results for solutions to SDEs in terms of the driving
vector fields can be traced back to the seminal work of Chen \cite{Ch}.
They have then been deeply analyzed in \cite{hu,St}, and also lye at the basis
of the rough path theory \cite{LQ}. We have chosen here to present these formulas
according to~\cite{Ba}, which is a recent and didactically useful account on the topic.

\smallskip

Recall that we are considering a $d$-dimensional fBm $B$ with $1/3<H<1/2$. According to Section \ref{sec:fbm}, this allows to construct some increments $\bb^{\bk}$ out of $B$ which can be seen as limits of Riemann iterated integrals over the simplex $\cs_k([s,t])$, as recalled at Proposition~\ref{prop:limit-iterated-intg}. Furthermore, one can solve equation (\ref{eq:fractional-sde}) under the conditions of Theorem~\ref{thm:Lip}.

\smallskip

Let us introduce some additional notation: let $\cv$ be the space of smooth bounded vector fields over $\R^m$. If $V\in\cv$, the vector field $\exp(V)\in\cv$ is defined by the relation $[\exp(V)](\xi)=\Psi_1(\xi)$, where $\{\Psi_{t}(\xi);\, t\ge 0\}$ is solution to the ordinary differential equation
\begin{equation}
\partial_{t} \Psi_{t}(\xi)= V\lp \Psi_{t}(\xi) \rp,
\quad \Psi_{0}(\xi)=\xi.
\end{equation}
The aim of Chen-Strichartz formula is to express the solution $y_t$ to equation (\ref{eq:fractional-sde}), for an arbitrary $t\in\ott$, as $y_t=[\exp(Z_t)](a)$ for a certain $Z_t\in\cv$.

\smallskip

To this purpose, let us give some more classical notations on vector fields: if $V,W\in\cv$, then the vector field $[V,W]\in\cv$ (called Lie bracket of $V$ and $W$) is defined by
\begin{equation*}
[V,W]^{i}=V^{l} \partial_{x_l} W^{i} - W^{l} \partial_{x_l} V^{i}.
\end{equation*}
Notice that this notion is usually introduced though the interpretation of $\cv$ as a set of first order differential operators. A Lie bracket of order $k$ can also be defined inductively for $k\ge 2$ by setting
\begin{equation*}
\lc U_1\cdots U_k \rc =
\lc \lc U_1\cdots U_{k-1}\rc,\, U_k \rc,
\end{equation*}
for $U_1\ldots U_k\in\cv$. With this notation in hand, our main assumption on the vector fields
$V_1,\ldots,V_d$ governing equation (\ref{eq:fractional-sde}) is the following:
\begin{hypothesis}\label{hyp:nilpotent-V}
The vector fields $V_1,\ldots,V_d$ are $n$-nilpotent for some  given positive integer $n$.
Namely, for any $(i_1,\ldots,i_n)\in\{1,\ldots,d\}^{n}$, we have $[V_{i_1}\cdots V_{i_n}]=0$.
\end{hypothesis}

\smallskip

We are now ready to state our formulation of Strichartz' identity,
for which we need two last notations: for $k\ge 1$, we call $\mathfrak{S}_k$
the set of permutations of $\{1,\ldots,k\}$. Moreover, for $\si\in\mathfrak{S}_k$,
write $e(\si)$ for the quantity ${\rm Card}(\{j\in\{1,\ldots,k-1\};\, \si(j)>\si(j+1)\})$.
Then the following formula is proven e.g. in~\cite{Ba,hu,St}:
\begin{proposition}\label{prop:strichartz}
Under the hypothesis of Theorem \ref{thm:Lip}, let $y$ be the solution to equation~
(\ref{eq:fractional-sde}). Assume Hypothesis~\ref{hyp:nilpotent-V} holds true, and
consider $t\in\ott$. Then $y_t=[\exp(Z_t)](a)$, where $Z_t$ can be expressed as follows:
\begin{equation*}
Z_t=\sum_{k=1}^{n-1}\sum_{i_1,\ldots,i_k=1}^{d}  \bv_{i_1,\ldots,i_k} \, \psi_{t}^{i_1,\ldots, i_k},
\quad\mbox{with}\quad
\psi_{t}^{i_1,\ldots, i_k} = \sum_{\si\in\mathfrak{S}_k}
\frac{(-1)^{e(\si)}}{k^2 {{k-1}\choose{e(\si)}}} \, \bb_{0t}^{\bk, i_{\tau(1)},\ldots, i_{\tau(k)}},
\end{equation*}
where we have set $\tau=\si^{-1}$ and $\bv_{i_1,\ldots,i_k}=[V_{i_1}\cdots V_{i_k}]$ in the formula above.
\end{proposition}

\smallskip

As a warmup to the computations below, we prove now that one can extend our inequality (\ref{control-sol-y}) thanks to Strichartz representation, covering the case of unbounded vector fields with bounded derivatives:
\begin{proposition}\label{prop:moments-y-t-nilpotent}
Suppose Hypothesis~\ref{hyp:nilpotent-V} holds true, and that the smooth vector fields $V_i$, $i=1, 2, \cdots, d$  have bounded derivative.   Assume moreover that all the Lie brackets of order
greater or equal to $2$ are bounded vector fields. Then the solution $y$ of equation (\ref{eq:fractional-sde}) admits moments of any order.
Namely, for any $m>1$ and any $T\in (0, \infty)$,
\begin{equation}\label{e.42}
\EE\lc \sup_{0\le t\le T} |y_t|^m\rc <\infty\,.
\end{equation}
\end{proposition}

\begin{proof} 
One can restate Proposition \ref{prop:strichartz} as follows: for any $t\le T$, the random variable $y_t$ can be expressed as $y_t=\phi_1$,  where $\phi_s^t:=\phi_s:\RR^d\rightarrow \RR^d$   satisfies
(for $t$ fixed)
\[
\partial_{s}\phi_s =  \sum_{k=1}^{n-1}
  \sum_{i_1,\ldots,i_k =1}^ d   \psi_{t}^{i_1,\ldots, i_k} \bv_{i_1,\ldots,i_k} (\phi_s)\,, \quad 0\le s\le 1\,,\quad \phi_0=a\,.
\]
Let us separate the first order integrals in this equation, in order to get
\begin{equation}
\partial_{s}\phi_s = \sum_{i=1}^d  V_i(\phi_s) B^{i}_{t}+
 \sum_{k=2}^{n-1}
  \sum_{i_1,\ldots,i_k =1}^ d   \psi_{t}^{i_1,\ldots, i_k} \bv_{i_1,\ldots,i_k} (\phi_s)\,, \quad \phi_0=a\,.
 \label{e.43}
\end{equation}
Since $V_i$, $i=1, 2, \cdots, d$,  have bounded derivatives and since
all the Lie brackets of order greater or equal to $2$ are bounded
we see that
\begin{eqnarray*}
\left|\partial_{s}\phi_s\right|
&\le&  \sum_{i=1}^d  \left|V_i(\phi_s)\right| |B^{i}_{t}|+
 \sum_{k=2}^{n-1}
  \sum_{i_1,\ldots,i_k =1}^ d   |\psi_{t}^{i_1,\ldots, i_k} | |\bv_{i_1,\ldots,i_k} (\phi_s)| \\
 &\le&  c_1 \left| \phi_s \right|   \sum_{i=1}^d  \sup_{0\le t\le T} |B^{i}_{t}|+ c_2
 \sum_{k=1}^{n-1}
  \sum_{i_1,\ldots,i_k =1}^ d   \sup_{0\le t\le T}  |\psi_{t}^{i_1,\ldots, i_k} |\,.
\end{eqnarray*}
Thus by Gronwall's lemma, we have
\begin{eqnarray*}
\left| \phi_s\right|
 &\le& c_2\left(
 \sum_{k=1}^{n-1}
  \sum_{i_1,\ldots,i_k =1}^ d   \sup_{0\le t\le T}  |\psi_{t}^{i_1,\ldots, i_k} | \right) \,
  \exp\left\{  c_1 \sum_{i=1}^d    \sup_{0\le t\le T}  |B^{i}_{t}|\right\}
  \,.
\end{eqnarray*}
This inequality holds true for all  $0\le s\le 1$.  Thus
\begin{eqnarray*}
\sup_{0\le t\le T} \left| y_t \right|
 &\le& c_2\left(
 \sum_{k=1}^{n-1}
  \sum_{i_1,\ldots,i_k =1}^ d   \sup_{0\le t\le T}  |\psi_{t}^{i_1,\ldots, i_k} | \right) \,
  \exp\left\{  c_1 \sum_{i=1}^d    \sup_{0\le t\le T}  |B^{i}_{t}|\right\}
  \,,
\end{eqnarray*}
which implies \eref{e.42}.
\end{proof}

\subsection{Malliavin derivative}
This subsection is devoted to enhance our Proposition \ref{prop:moments-y-t-nilpotent}, and prove that the Malliavin derivative of $y_t$ has also bounded moments of any order in our particular nilpotent situation. Notice once again that the boundedness of moments of the Malliavin derivative is still an open problem for rough differential equations in the general case. We refer to Section \ref{sec:malliavin-derivatives} for notations on Malliavin calculus.
\begin{theorem} \label{t.4.4}  Let the vector fields $V_i$, $1=1, 2, \cdots, d$ be
smooth with all derivatives bounded, satisfying Hypothesis \ref{hyp:nilpotent-V}.   Assume that all the Lie
brackets of order greater or equal to $2$ are  constants.
Then  the Malliavin derivative $y_t$ has moments of any order.
More precisely, for any $q>1$ and $T\in (0, \infty)$,
\begin{equation}
\EE\lc \sup_{  0\le u\le t\le T}  \left|  D_u y_t\right| ^q \rc<\infty\,.
\end{equation}
\end{theorem}

\begin{proof}

Go back to our representation \eref{e.43}, which can easily be differentiated in the Malliavin calculus sense in order to obtain
\begin{eqnarray*}
\partial_{s} D_u \phi_s
&=&
 \sum_{i=1}^d  \nabla V_i(\phi_s)  B^{i}_{t} \, D_u \,  \phi_s +
 \sum_{k=2}^{n-1}
  \sum_{i_1,\ldots,i_k =1}^ d   \psi_{t}^{i_1,\ldots, i_k} \, \nabla\bv_{i_1,\ldots,i_k}  (\phi_s) \, D_u \phi_s\\
 &&\quad +  \sum_{i=1}^d  V_i (\phi_s) \1_{[0,t)}^{[i]}(u)+
 \sum_{k=2}^{n-1}
  \sum_{i_1,\ldots,i_k =1}^ d   D_u \psi_{t}^{i_1,\ldots, i_k} \bv_{i_1,\ldots,i_k}   (\phi_s) \, ,
\end{eqnarray*}
where we have set $\nabla\bv_{i_1,\ldots,i_k}$ for the (matrix valued) gradient of $\bv_{i_1,\ldots,i_k}$, and where we recall that the notation $\1_{[0,t)}^{[i]}$ has been introduced at Section \ref{sec:malliavin-derivatives}.

\smallskip

Since we assume that all the Lie brackets of order
greater or equal to $2$ of the vector fields $V_i$ are constant vector fields, it is easily checked that
\begin{equation}\label{eq:mall-deriv-order2}
\partial_{s} D_u \phi_s
=
 \sum_{i=1}^d  \nabla V_i(\phi_s)  B^{i}_{t} D_u \phi_s  
 +  \sum_{i=1}^d  V_i (\phi_s) \1_{[0,t)}^{[i]}(u)+
 \sum_{k=2}^{n-1}
  \sum_{i_1,\ldots,i_k =1}^ d   D_u \psi_{t}^{i_1,\ldots, i_k} \bv_{i_1,\ldots,i_k}   (\phi_s) \,.
\end{equation}
Therefore there exist two positive constants $c_1,c_2$ such that
\begin{eqnarray*}
\left| \partial_{s} D_u \phi_s\right|
&\le&
 c_1 \sum_{i=1}^d     |B^{i}_{t}| |D_u \phi_s  | \\
 &&\quad +  c_2 \sum_{i=1}^d  [1+ |\phi_s|] \, \1_{[0,t)}^{[i]}(u) +
 \sum_{k=2}^{n-1}
  \sum_{i_1,\ldots,i_k =1}^ d  \left|  D_u \psi_{t}^{i_1,\ldots, i_k}\right|| \bv_{i_1,\ldots,i_k}   (\phi_s)| \,.
\end{eqnarray*}
By Gronwall's lemma we obtain
\begin{multline*}
\sup_{0\le  s\le 1, 0\le u\le t\le T}  \left|  D_u \phi_s\right|
\le c_2
\exp\left\{  c_1 \sum_{i=1}^d     \sup_{0\le t\le T}  |B^{i}_{t}| \right\}
  \times \Bigg\{  \sum_{i=1}^d  \left[ 1+\sup_{0\le   s\le 1} |\phi_s|\right] \\
\times \lp  1+ \sup_{0\le t\le T}
 \sum_{k=2}^{n-1}
  \sum_{i_1,\ldots,i_k =1}^ d  \sup_{0\le  s\le 1, 0\le u\le t\le T}
  \left|  D_u \psi_{t}^{i_1,\ldots, i_k}\right|| \bv_{i_1,\ldots,i_k}   (\phi_s)| \rp \Bigg\} \,.
\end{multline*}
Thus
\begin{multline*}
\sup_{  0\le u\le t\le T}  \left|  D_u y_t\right|
\le c_2
\exp\left\{  c_1 \sum_{i=1}^d     \sup_{0\le t\le T}  |B^{i}_{t}| \right\}  \\
   \times  \Bigg\{  \sum_{i=1}^d  \left[ 1+\sup_{0\le   t\le T} |y_t|\right]
  \lp 1 +
 \sum_{k=2}^{n-1}
  \sum_{i_1,\ldots,i_k =1}^ d  \sup_{0\le  s\le 1, 0\le u\le t\le T}
  \left|  D_u \psi_{t}^{i_1,\ldots, i_k}\right| \rp \Bigg\} \, ,
\end{multline*}
which ends the proof easily by boundedness of moments for $y_t$, $D_u \psi_{t}^{i_1,\ldots, i_k}$ and $B^i_t$.

\end{proof}

\begin{example} A classical  example of nilpotent vector fields in $\RR^3$
is due to Yamato \cite{Ya}.  Let us check that this example fullfils our standing assumptions. Indeed, the example provided in \cite{Ya} is the following:
\[
A_1=0\,,\quad A_2=\frac{\partial } {\partial x_1}+2x_2 \frac{\partial } {\partial x_3}\,,
\quad \hbox{\rm and}\quad
A_3=\frac{\partial } {\partial x_2}-2x_1\frac{\partial } {\partial x_3} \,.
\]
Then
\begin{eqnarray*}
[A_2, A_3]&=&-4\frac{\partial } {\partial x_3}\,,\quad
[[A_2,A_3], A_2]=[[A_2, A_3], A_3]=0\,.
\end{eqnarray*}
It is thus readily checked that the conditions of Theorem \ref{t.4.4} are met for these vector fields. Moreover, in this particular case the solution to equation (\ref{eq:fractional-sde}) is explicit and we have
\begin{equation*}
y_1^1=y_1+B_t^2, \qquad
y_t^2=y_2+B_t^3, \qquad
y_t^3=y_3+2\lp \bb^{\2,32}_{0t}- \bb^{\2,23}_{0t}\rp\, ,
\end{equation*}
if the solution starts from the initial condition $(y_1,y_2,y_3)$.
Interestingly enough, though the solution is explicit here, the smoothness of the density of $y_t$ is not immediate and we recover here the results of \cite{Dr}.
\end{example}

\subsection{Stochastic flows}
The probabilistic proof of the smoothness of density for diffusion processes originally given by Malliavin \cite{Ma} heavily relies on stochastic flows methods and their relationship with stochastic derivatives. We now establish those relations for SDEs driven by a fractional Brownian motion.

\smallskip

To this aim, denote by $y^{s,a}$ the solution to equation \eref{eq:fractional-sde} starting from the initial condition $y_s=a$ at time $s$:
\begin{equation}\label{eq:fractional-sde-s}
dy_t^{s, a} =\sum_{i=1}^{d}V_i(y_t^{s, a}) \, dB_t^{i},
\quad t \in \lc s, T \rc, \qquad  y_s^{s, a}=a\,.
\end{equation}
The above equation gives  rise to a family of smooth nonlinear mappings
$\Phi_{s,t}:\RR^m\rightarrow \RR^m$, $0\le s\le t\le T$,  determined by
$\Phi_{s,t}(a):= y_t^{s, a}$, and the family $\{\Phi_{s,t}; \,\le s\le t\le T\}$  has the following  flow property (we refer e.g. to \cite{FV} for the properties of flows driven by rough paths quoted below):
\[
\Phi_{s, t}=\Phi_{u, t} \circ \Phi_{s,u}\,,  \quad \forall \ 0\le s\le u\le t\le T\,.
\]
Let $J_{s,t}$  denote the gradient of the nonlinear mapping $\Phi_{s,t}$ with respect to the initial condition.
Then the family  $\{J_{s,t}; \,\le s\le t\le T\}$ also satisfies the relation $J_{s,t}=J_{u,t} J_{s,u}$ for  $\ 0\le s\le u\le t\le T$.
In addition, the map $J_{s, t}$ is invertible, and we have $J_{s,t}=J_{0,t}J_{0,s}^{-1}$.
The equation followed by $J_{0,t}$ is obtained by differentiating formally equation \eref{eq:fractional-sde-s} with respect to the initial value $a$, which yields
\[
dJ_{0,t}=\sum_{i=1}^{d} \nabla V_i(y_t) J_{0, t}  \, dB_t^{i}\,, \quad J_{0, 0}=I\, .
\]
By applying the rules of differential calculus for rough paths, we also get that $J_{0,t}^{-1}$ is solution to the following equation:
\begin{equation}\label{eq:}
dJ_{0,t}^{-1}=-\sum_{i=1}^{d} \nabla V_i(y_t) J_{0, t}^{-1}  \, dB_t^{i}\,, \quad J_{0, 0}^{-1}=I\,.
\end{equation}
We have thus ended up with two linear equations for the derivatives of the flow. In our nilpotent case, we are able thus able to bound these derivatives along the same lines as for Theorem \ref{t.4.4}:
\begin{theorem} \label{thm:moments-jacobian} 
Let the vector fields $V_i$, $1=1, 2, \ldots, d$ be
smooth with all derivatives bounded, satisfying Hypothesis \ref{hyp:nilpotent-V}.   Assume that all the Lie
brackets of order greater or equal to $2$ are  constant.
Then  the Jacobian $J_{0, t}$ and its inverse $J_{0, t}^{-1}$ have moments of any order: for any $q\ge 1$ and $T\in (0, \infty)$,
\begin{equation}\label{e.46}
\EE\lc\sup_{  0 \le t\le T}  \left|  J_{0, t} \right| ^q\rc<\infty ,
\quad\mbox{and}\quad
\EE\lc \sup_{  0 \le t\le T}  \left|  J_{0, t}^{-1}  \right| ^q\rc<\infty\,.
\end{equation}
\end{theorem}

\begin{proof} 
As mentioned in the proof of Proposition \ref{prop:moments-y-t-nilpotent}, one can write $\Phi_{0, t}(a)=\exp(Z_t)(a)=\phi_1(a)$,  where $\phi_1(a)$ satisfies
\eref{e.43}. Thus if we  introduce   $\tilde J_s=\nabla \phi_s$,  then  $J_{0, t}=\tilde J_{1}$ where
  $\tilde J_s $ satisfies
\begin{equation*}
\partial_{s}\tilde J_s = \sum_{i=1}^d  B^{i}_{t} \, \nabla V_i(\phi_s) \, \tilde J_s +
 \sum_{k=2}^{n-1}
  \sum_{i_1,\ldots,i_k =1}^ d   \psi_{t}^{i_1,\ldots, i_k} \, \nabla \bv_{i_1,\ldots,i_k} (\phi_s) \, \tilde J_s\,, \quad \tilde J_0=I\,.
\end{equation*}
The first part of \eref{e.46} is thus proved following the steps of Proposition \ref{prop:moments-y-t-nilpotent}.

\smallskip

As far as the second part of \eref{e.46} is concerned, observe that $J_{0, t}^{-1}=\bar J_1$,  where $\bar J_s$ is inverse of $\tilde J_s$.
It is clear that  $\bar J_s$ satisfies
\begin{equation*}
\partial_{s}\bar J_s = -\sum_{i=1}^d  \nabla V_i(\phi_s) \bar J_s B^{i}_{t}-
 \sum_{k=2}^{n-1}
  \sum_{i_1,\ldots,i_k =1}^ d   \psi_{t}^{i_1,\ldots, i_k} \nabla \bv_{i_1,\ldots,i_k} (\phi_s) \bar J_s\,, \quad \tilde J_0=I\,.
\end{equation*}
Once again, the methodology of Proposition \ref{prop:moments-y-t-nilpotent} easily yields our claim.

\end{proof}

\begin{corollary}\label{cor:bnd-Dy-J-Z-U}
Under the same assumptions as in Theorem \ref{thm:moments-jacobian}, the following holds true:

\smallskip

\noindent\emph{(i)}
For any $0<\ga<H$ and $q\ge 1$ we have
\begin{equation}\label{eq:bnd:jacobian-holder-norm}
\be\lc \|Dy_t\|_{\ga,\infty}^{q} \rc + \be\lc \|J_{0,\cdot}\|_{\ga,\infty}^{q} \rc + \be\lc \|J_{0,\cdot}^{-1}\|_{\ga,\infty}^{q} \rc
<c_{T,q},
\end{equation}
for a finite constant $c_{T,q}$.

\smallskip

\noindent \emph{(ii)}
As a consequence, inequality \eref{eq:bnd:jacobian-holder-norm} also holds true when the $\|\cdot\|_{\ga,\infty}$ norms are replaced by norms in $\ch$, where $\ch$ has been defined at Section \ref{sec:fct-spaces}.

\smallskip

\noindent \emph{(iii)}
For any smooth bounded vector field $U$ on $\R^m$ and $t\in\ott$, set $Z^{U}_t:=\langle J_{0,t}^{-1}U(y_t), \, \eta\rangle$. Then $Z^U$ is a controlled process, and  satisfies the inequality $\be[\cn^q[Z^U;\, \cq^\ga(\R^m)]]\le c_{T,q}$ for any $q\ge 1$ and a finite constant $c_{T,q}$.
\end{corollary}

\begin{proof}
Going back to equation \eref{eq:mall-deriv-order2}, it is readily checked that all the terms $u\mapsto D_u \psi_{t}^{i_1,\ldots, i_k}$ are $\cac_1^\ga$-H\"older continuous on $[0,t]$ for any $\ga<H$, since the elements $\psi_{t}^{i_1,\ldots, i_k}$ are nice multiple integrals with respect to $B$. Moreover, we have
\begin{equation*}
\be\lc \|D \psi_{t}^{i_1,\ldots, i_k}\|_{\ga,\infty}^{q} \rc < \infty,
\end{equation*}
for any $m\ge 1$. This easily yields $\be[ \|D y_{t}^{i_1,\ldots, i_k}\|_{\ga,\infty}^{q} ] < \infty$ by a standard application of Gronwall's lemma, as in the proof of Theorem \ref{t.4.4}.

\smallskip

Our second assertion stems from the fact that one can choose $1/2-H<\ga<H$, since $H>1/3$. For such $\ga$ we have $\cac_1^\ga\subset \ch$, which ends the proof.

\smallskip

Finally, our claim (iii) derives from the fact that the equation governing $Z^{U}$ is of the following form:
\begin{equation}\label{eq:sde-Z-U}
Z_t^{U}= \lla \eta, \, U(a)\rra+ \sum_{j=1}^{d} \iot Z_s^{[U,V_j]} dB_s^j.
\end{equation}
The process $Z^{U}$ can thus be decomposed as a controlled process as in Section \ref{sec:rdes}, and since we already have estimates for $J_{0,t}^{-1}$ and $y_t$, the bound on $\be[\cn^q[Z^U;\, \cq^\ga(\R^m)]]$ follows easily.

\end{proof}

\subsection{Proof of Theorem \ref{thm:smooth-density}}
As mentioned in the introduction, once we have shown the integrability of the Malliavin derivative and proved a Norris type lemma, the proof of our main theorem goes along classical lines. We have chosen to follow here the exposition of~\cite{Ha11}, to which we refer for further details.

\smallskip

\noindent
\textit{Step 1: Reduction to a lower bound on H\"older norms.}
Recall that the process $Z^U$ has been defined for any smooth vector field $U$ in Corollary \ref{cor:bnd-Dy-J-Z-U}. For any $p\ge 1$, our first goal is to reduce our problem to the existence of a constant $c_p$  such that
\begin{equation}\label{eq:low-bnd-C-alpha2}
\bp\lp  \|Z^{V_k}\|_{\al,\infty} \le \ep \rp \le c_p \ep^p,
\end{equation}
for $1\le k\le d$, a given $\al\in(1/3,1/2)$, all $\ep\in(0,1)$ and where we observe that all the norms below are understood as norms on $\ott$.

\smallskip

Indeed, according to \cite[Relation (4.9)]{Ha11} the smoothness of density can be obtained from the estimate 
\begin{equation*}
\bp\lp  \|Z^{V_k}\|_{\ch} \le \ep \rp \le c_p \ep^p,
\end{equation*}
where we recall that $\ch$ has been defined at Section \ref{sec:fct-spaces}. Furthermore, we have
\begin{equation*}
\bp\lp  \|Z^{V_k}\|_{\ch} \le \ep \rp \le
\bp\lp  \|Z^{V_k}\|_{L^2} \le \ep \rp \le 
\bp\lp  \|Z^{V_k}\|_{L^1} \le \ep \rp.
\end{equation*}
It is thus sufficient for our purposes to check 
\begin{equation}\label{eq:low-bnd-L1}
\bp\lp  \|Z^{V_k}\|_{L^1} \le \ep \rp \le c_p \ep^p.
\end{equation}

\smallskip

In order to go from \eref{eq:low-bnd-L1} to \eref{eq:low-bnd-C-alpha2}, let us use our interpolation bound \eref{eq:interpol-ineq} in the following form: for any $0<\eta<1$ and $0<\al<\rho<H$ we have
\begin{equation*}
\|b\|_{L^1} \ge \eta^{1/(\rho-\al)} \lp \|b\|_{\al,\infty}-C_{\al, \rho} \eta \, \|b\|_{\rho,\infty} \rp.
\end{equation*}
Take now $\delta\in(0,1)$ to be fixed later on and  $\eta^{1(\rho-\al)}=\ep^{1-\delta}$, that is $\eta=\ep^{(\rho-\al)(1-\delta)}$. Then
\begin{equation}\label{eq:low-bnd-L1-2}
\bp\lp  \|Z^{V_k}\|_{L^1} \le \ep \rp \le 
\bp\lp  \|Z^{V_k}\|_{\al,\infty} \le 2 \ep^\delta \rp + R,
\quad\mbox{where}\quad
R=\bp\lp \|Z^{V_k}\|_{\rho,\infty} \le  \frac{1}{4c \, \ep^{\nu}}\rp,
\end{equation}
with $\nu=\rho-\al-(1-(\rho-\al))\delta$. Choose now $\delta$ small enough, so that $\nu>0$. Since $\|Z^{V_k}\|_{\rho,\infty}$ admits moments of any order according to Corollary \ref{cor:bnd-Dy-J-Z-U}, it is easily checked that $R$ can be made smaller than any quantity of the form $c_q \ep^{q}$. It is thus sufficient to prove \eref{eq:low-bnd-C-alpha2} in order to get the smoothness of density for $y_t$.

\smallskip

\noindent
\textit{Step 2: An iterative procedure.}
For $l\ge 1$ and $x\in\R^m$, let $\cv_l(x)$ be the vector space generated by the Lie brackets of order $l$ of our vector fields $V_1,\ldots,V_d$ at point $x$: 
\begin{equation*}
\cv_l(x)={\rm Span}\lcl [V_{k_1}\cdots V_{k_j}](x); \, j\le l, \, 1\le k_1,\ldots, k_j\le d \rcl.
\end{equation*}
We assume that the vector fields are $\ell$-hypoelliptic for a given $\ell>0$, which can be read as $\cv_\ell(x)=\R^m$ for any $x\in\R^m$. In order to start our induction procedure, we set $\al_1=\al$, so that we have to prove $\bp(  \|Z^{V_k}\|_{\al_1,\infty} \le \ep ) \le c_p \ep^p$.

\smallskip

Recall that $Z^{V_k}$ satisfies the relation
\begin{equation*}
Z_t^{V_k}= \lla \eta, \, V_k(a)\rra+ \sum_{j=1}^{d} \iot Z_s^{[V_k,V_j]} dB_s^j.
\end{equation*}
Thus Proposition \ref{prop:norris-lemma} asserts that for any $1/3<\al_2<\al_1<H$ there exists $q_2>0$ such that
\begin{equation*}
\bp\lp  \lp \|Z^{V_k}\|_{\al_1,\infty} \le \ep \rp
\cap \lp \cup_{j=1}^{d} \lp \|Z^{[V_{k_1},V_{k_2}]}\|_{\al_2,\infty}> \ep^{q_2} \rp \rp \rp \le c_p \ep^p.
\end{equation*}
Relation \eref{eq:low-bnd-C-alpha2} is thus implied by
\begin{equation*}
\bp\lp  \lp \|Z^{V_k}\|_{\al_1,\infty} \le \ep \rp
\cap \lp \cap_{j=1}^{d} \lp \|Z^{[V_{k_1},V_{k_2}]}\|_{\al_2,\infty}\le \ep^{q_2} \rp \rp \rp \le c_p \ep^p.
\end{equation*}
Iterating this procedure we end up with the following claim to prove: $\cb_{\ell}(\ep)\le c_p \ep^p$ for all $\ep\in(0,1)$, with
\begin{equation*}
\cb_{\ell}(\ep)=
\bp\lp  \|Z^{V_k}\|_{\al_1,\infty} \le \ep, \,  \|Z^{[V_{k_1},V_{k_2}]}\|_{\al_2,\infty}\le \ep^{q_2},
\ldots,  \|Z^{[V_{k_1}\cdots V_{k_\ell}]}\|_{\al_\ell,\infty}\le \ep^{q_\ell}\rp,
\end{equation*}
where the intersection above extends to all possible combinations $1\le k_1,\ldots,k_\ell\le d$, and where $1/3<\al_\ell<\cdots<\al_1<H$.

\smallskip

Going back now to the very definition of  $Z^U$ as $Z^{U}_t=\langle J_{0,t}^{-1}U(y_t), \, \eta\rangle$, it is readily checked that
\begin{equation*}
\cb_{\ell}(\ep) \le 
\bp\lp  \lla \eta, \, V_{k_1}(a) \rra \le \ep, \ldots, \lla \eta, \, [V_{k_1}\cdots V_{k_\ell}](a) \rra \le \ep^{q_\ell}
\rp.
\end{equation*}
Owing to the fact that $\cv_\ell(a)=\R^m$, we thus have $\cb_{\ell}(\ep)=0$ for $\ep$ small enough, which ends the proof.

\medskip

\noindent
\textbf{Acknowledgement:} We would like to thank Ivan Nourdin for pointing us out the use of Hermite polynomials in the computation of 4\textsuperscript{th} order variations (see proof of Proposition~\ref{prop:norris-lemma} Step 2).

\end{document}